\documentclass[reqno,a4paper,10pt]{amsart}
\usepackage[a4paper, centering]{geometry}
\usepackage[utf8]{inputenc} % allow utf-8 input
\usepackage[english]{babel}
\usepackage[T1]{fontenc}
\usepackage{lmodern}
\usepackage{caption}
\usepackage{subcaption}
\usepackage{enumerate}
\usepackage{amssymb, amsmath}
\usepackage{mathrsfs,dsfont}
\usepackage{amscd}
\usepackage{bbm}
\usepackage{amsthm}
\usepackage{cite}
\usepackage{esint }
\usepackage[mathcal]{euscript}
\usepackage[active]{srcltx}
\usepackage{verbatim}
\usepackage[colorlinks,linkcolor={blue},citecolor={blue},urlcolor={black}]{hyperref}
\usepackage[]{changebar}
\usepackage{xcolor}
\usepackage{mathtools}
\usepackage{url}            % simple URL typesetting
\usepackage{booktabs}       % professional-quality tables
\usepackage{amsfonts}       % blackboard math symbols
\usepackage{amsmath}
\DeclareRobustCommand{\EuScript}[1]{\ifmmode\mathcal{#1}\else$\mathcal{#1}$\fi}
\usepackage{nicefrac}       % compact symbols for 1/2, etc.
\usepackage{microtype}      % microtypography
\usepackage{lipsum}
\usepackage{fancyhdr}       % header
\usepackage{graphicx}       % graphics
\graphicspath{{media/}}     % organize your images and other figures under media/ folder
\usepackage{bookmark} %(fa vedere le formule)
\usepackage[a4paper, centering]{geometry}
\usepackage{epstopdf}
\usepackage{amscd}
\usepackage{color}
\usepackage{graphicx}

\usepackage{yfonts}
\usepackage{fullpage}
\usepackage{comment}
\usepackage[braket,qm]{qcircuit} % quantum circuits
\usepackage{ stmaryrd } % per un commando che chiamo lket dopo
\newtheorem{thm}{Theorem}[section]

\newtheorem*{thm*}{Theorem}

\newtheorem{lem}{Lemma}[section]

\newtheorem{prop}{Proposition}[section]

\newtheorem*{prop*}{Proposition}
\newtheorem{ass}{A}

\theoremstyle{definition}
\newtheorem{defn}{Definition}[section]

\newtheorem{ex}{Example}[section]

\theoremstyle{remark}
\newtheorem{rem}{Remark}[section]

\numberwithin{equation}{section}
\def\N{{\mathbb N}}

\def\R{{\mathbb R}}

\def\L{{\mathcal L}}

\def\H{{\mathcal H}}

\def\E{{\mathbb E}}

\def\X{{\mathcal X}}
\def\Y{{\mathscr Y}}

\def\I{{\mathbbm{1}}}
\def\cP{{\mathscr{P}}}
\def\norma #1{\left\lVert #1 \right\rVert}
\def\B{{\mathscr{B}}}
\def\P{{\mathbb{P}}}

\def\de{{\rm d}}

\def\w{{\rm W}}
\def\F{{\mathscr{F}}}

\def\M{{\mathcal{M}}}
\def\nm #1{ \left\langle #1 \right\rangle}

\newcommand{\hol}[1]{\"{#1}}

\def\W{{\mathcal W}}
\def\BL{{\rm BL}}

\def\E{{\mathbb E}}
\def\F{{\mathscr F}}
\def\B{{\mathscr B}}
\def\H{{\mathscr H}}
\def\I{{\mathcal I}}
\def\cS{{\mathfrak S}}

\def\y{{\boldsymbol y}}
\def\vH{{\boldsymbol H}}
\def\bR{{\boldsymbol \R}}
\def\bB{{\boldsymbol B}}
\def\bQ{{\boldsymbol Q}}
\def\X{{\boldsymbol X}}
\def\by{{\boldsymbol Y}}

\def\blambda{{\boldsymbol{\lambda}}}
\def\Tt{{\mathbb{R}}}
\def\R{{\mathcal{R}}}
\def\cY{{\mathfrak Y}}

\def\w2{\stackrel{2}{\rightharpoonup}}

\def\strong2{\stackrel{2}{\rightarrow}}
\definecolor{viola}{rgb}{0.3,0,0.7}
\definecolor{ciclamino}{rgb}{0.5,0,0.5}
\definecolor{rosso}{rgb}{0.8,0,0}

\makeatletter
\def\cleardoublepage{\clearpage\if@twoside \ifodd\c@page\else
\hbox{}
\thispagestyle{empty}
\newpage
\if@twocolumn\hbox{}\newpage\fi\fi\fi}
\makeatother
\title[MFG]{Gradient Mean-Field Dynamics with Measure-Valued States: Well-Posedness, Chaos, and Long-Time Stability}

\author[A. Melchor Hernandez]{Anderson Melchor Hernandez}
\address[A. Melchor Hernandez]{Dipartimento di Matematica,
Università di Bologna, Via Zamboni 33,
40126, Bologna, Italy.}
\email{anderson.melchor@unibo.it}
\date{\today}

\begin{document}

\begin{abstract}
We study a stochastic mean-field interacting particle system whose state space is 
$\Y = \Tt^d \times \cP(U)$, where the first component represents a spatial variable and the second one is a probability measure over a compact metric space $U$. 
The dynamics are driven by locally Lipschitz drift operators: the spatial component evolves according to a Brownian diffusion, while the measure-valued component is perturbed by a projected cylindrical noise acting in the Arens--Eells space.

We first establish existence and uniqueness of strong solutions for both the $N$-particle system and the associated nonlinear McKean--Vlasov equation under locally Lipschitz and linear growth assumptions on the drift coefficients. 
We then prove propagation of chaos: as $N\to\infty$, the empirical measure converges in expectation in Wasserstein--1 distance towards the unique McKean--Vlasov solution.

Further, we investigate exponential convergence of the nonlinear McKean--Vlasov dynamics towards a unique invariant measure. 
\end{abstract}
\maketitle
\tableofcontents

\section{Introduction}

Mean-field interacting particle systems provide a powerful framework for describing the collective behavior of large populations of agents subject to deterministic interactions and stochastic fluctuations. Such models arise naturally in kinetic theory, statistical mechanics, interacting diffusions, and in the mathematical analysis of socio-economic systems, crowd dynamics, and multi-agent systems. In particular, kinetic and mean-field approaches have provided a systematic framework for the derivation of macroscopic models from microscopic interaction rules in crowd and swarming dynamics \cite{bellomo2008}. More recently, mean-field techniques have also played a central role in the analysis of learning dynamics and neural-type systems \cite{bellomo2012,pham2018bellman,robert2016dynamics}.
From a modeling perspective, our approach is conceptually related to the active particle framework proposed by 
\cite{bellomo2011}. In that setting, the dynamics of a large population is derived from interaction rules prescribed at the level of individual agents. A statistical description of the system is then introduced, 
leading to a kinetic equation governing the evolution of the population density. Similarly, in our framework, we start from a system of interacting stochastic differential equations describing the microscopic behavior 
of each agent. The collective dynamics emerges by considering the empirical measure of the system and passing to the limit as the number of agents tends to infinity, which yields a nonlinear McKean--Vlasov equation. However, unlike the classical kinetic PDE formulation, our model involves a measure-valued internal state for each particle and evolves in an infinite-dimensional functional space.

In this work, we investigate a stochastic interacting particle system whose state space is $\Y = \Tt^d \times \cP(U)$, where the first component represents a spatial state variable evolving on $\Tt^d$, and the second component is a probability measure over a compact set $U$. The dynamics couple a spatial diffusion with a measure-valued component evolving in a tangent-like space of $\cP(U)$ and perturbed by a generalized noise in the sense of Mamporia \cite{mamporia2004,mamporia2015}. The interaction between particles is of mean-field type and depends on the empirical distribution of the full system \cite{sirignano2022mean,Chaintron_2022a,Chaintron_2022b,sznitman1991topics}.

More precisely, we consider a system of $N$ particles $(X^i(t),\lambda^i(t))_{i=1}^{N}$ solving a coupled stochastic differential system with locally Lipschitz drift fields satisfying linear growth conditions. The associated empirical measure is defined by
\begin{align}
\mu^N(t)\coloneqq \frac1N \sum_{i=1}^N \delta_{(X^i(t),\lambda^i(t))}.    
\end{align}

The corresponding nonlinear McKean--Vlasov limit equation reads
\begin{align}\label{sistema1}
\begin{cases}
\de Y(t) = H(Y(t),\Lambda(t),\mu(t))\,\de t + \sqrt{2\sigma_x}\,\de B(t), \\
\de \Lambda(t) = \R(Y(t),\Lambda(t),\mu(t))\,\de t 
+ \sqrt{2\sigma_\lambda}\,\Pi_{\Lambda(t)}\,\de W^Q(t),
\end{cases}    
\end{align}
where $\mu(t)={\rm Law}(Y(t),\Lambda(t))$. The first equation describes the spatial evolution on $\Tt^d$, driven by the velocity field $H$ and a $d$-dimensional Brownian motion. The second equation governs the evolution of the measure component $\Lambda(t)$ and involves the velocity field $\R$ together with a generalized cylindrical noise $W^Q(t)$ acting in the Arens--Eells space $\F(U)$ which is defined as the span of $\cP(U)$ with respect to the bounded Lipschitz norm \cite{Arens-Eells,AFMS,fodoroptimal}. Several applications have been proposed for systems of the form \eqref{sistema1}. One example is the follower--leader model, aimed at describing the collective behavior of a controlled population $Y(t)$ influenced by a leader population distributed according to $\Lambda(t)$ \cite{almi2023opt,Ascione-Castorina-Solombrino}. In \cite{d2025large}, this framework was further developed to analyze a neuronal system in which the random variable $Y(t)$ represents the evolution of the membrane potential of a neuron receiving inputs from other neurons in the network. We emphasize that the model \eqref{sistema1} differs from those considered in \cite{d2025large,baldi2025well,Ascione-Castorina-Solombrino} due to the presence of the generalized noise $W^Q(t)$. Roughly speaking, this object mimics the properties of a Wiener process when $Q$ is a nuclear operator; however, its definition is formulated via duality, allowing us to consider cylindrical operators \cite{Da-Prato-SDE,2022lunardi}. Unlike the previously studied models, the introduction of this term may destroy the positivity of $\Lambda(t)$. To overcome this issue, we project the stochastic perturbation onto the tangent space at $\Lambda(t)$ through a suitable operator, denoted by $\Pi_{\Lambda(t)}$. This projection mechanism, combined with structural growth conditions on $Q$ and a martingale argument, ensures that $\Lambda(t)$ remains in $\cP(U)$ almost surely \cite{revuz2013continuous}.
We stress that the functional framework underlying the measure component relies on the structure of $\F(U)$, which is the canonical predual of the space of Lipschitz functions vanishing at a reference point. This dual characterization clarifies the compatibility between the bounded-Lipschitz norm on $\F(U)$ and the Wasserstein-$1$ distance (see \eqref{disutile} below). In particular, duality with Lipschitz functions naturally leads to estimates in $\W_1$, which can in turn be controlled by the $\BL$ norm \cite{MS2020,AFMS,baldi2025well,d2025large}.

\medskip
\textbf{Contributions.}
Our analysis contributes to the mathematical theory of mean-field limits in a infinite-dimensional setting \cite{Ascione-Castorina-Solombrino,ascione2020optimal,ascione2023deterministic,d2025large}. On the one hand, the measure-valued component of each particle evolves in the Arens--Eells space, and is driven by a projected cylindrical noise which is not assumed to be trace-class as usually done for the construction of a stochastic integral \cite{2022lunardi,Da-Prato-SDE}. This goes beyond classical McKean--Vlasov dynamics with measure-dependent drifts, where the noise acts only on a finite-dimensional state and positivity is preserved by construction \cite{baldi2025well,d2025large}. In our setting, a key technical difficulty stems from the fact that the noise does not preserve positivity a priori. In contrast to classical McKean--Vlasov models, where the law of each particle is automatically a probability measure and the stochastic perturbation does not act directly on the measure-valued coordinate, here the noise may drive $\Lambda(t)$ out of $\cP(U)$ \cite{zhang2026}. A central contribution of this work is to show that, despite this, the combination of the projection onto a suitable tangent space, structural conditions on the covariance operator $Q$, and a martingale argument yields almost-sure preservation of positivity for $\Lambda(t)$ \cite{Billingsley-Convergence,revuz2013continuous}. This allows us to treat a broader class of covariance operators than in previous works, while keeping the dynamics measure-valued. On the other hand, we obtain both well-posedness and propagation of chaos under purely local Lipschitz and linear growth conditions, and we establish exponential convergence to equilibrium under a natural gradient and uniform convexity structure, without invoking curvature or displacement convexity assumptions in the sense of the metric-gradient-flow theory \cite{ambrosio2014calculus,sznitman1991topics,meleard2006asymptotic}. We summarize our three principal results below.

\begin{thm}[Informal statement of \autoref{importantprop}]
Under locally Lipschitz and linear growth assumptions on the drift fields $H$ and $\R$, together with suitable structural conditions on the operator $Q$, the nonlinear McKean--Vlasov system \eqref{sistema1} admits a unique càdlàg strong solution $(Y(t),\Lambda(t)) \in \Y$ for all $t\in[0,T]$. Moreover, the $N$-particle system admits a unique strong solution on $[0,T]$.
\end{thm}

This result provides the foundation for the analysis of propagation of chaos.

\begin{thm}[Informal statement of \autoref{lem:conv_muN_detailed}]
Let $\mu^N(t)$ denote the empirical measure of the particle system and let $\mu(t)$ be the unique solution of the McKean--Vlasov equation. Then, as $N\to\infty$,
\begin{align}
\sup_{t\in[0,T]} \mathbb{E}\big[ \W_1(\mu^N(t),\mu(t)) \big]
\longrightarrow 0.
\end{align}
In particular, the particle system is chaotic with limit law $\mu(t)$.
\end{thm}

The previous result holds on finite time intervals. However, under stronger structural assumptions on the velocity fields $H$ and $\R$, it is possible to extend the analysis to the long-time regime.

\begin{thm}[Informal statement of \autoref{thm:exponentialconvergence}]
Under additional assumptions imposing a gradient structure and uniform convexity of the drift fields, the nonlinear McKean--Vlasov dynamics admits a unique invariant measure $\mu_\infty$. Moreover, the flow converges exponentially fast:
\begin{align}
\W_1(\mu(t),\mu_\infty)
\le e^{-\kappa t}\W_1(\mu(0),\mu_\infty),
\end{align}
where $\mu(0)$ denotes the law of the initial condition of \eqref{sistema1} and $\kappa>0$.
\end{thm}
As a consequence, we obtain uniform-in-time stability estimates for the mean-field limit, which transfer to the particle system.

\medskip

The paper is organized as follows. In Section~\ref{sec:preli}, we introduce the notation, the functional setting, and recall the notion of generalized stochastic integration \cite{mamporia2004}. In Section~\ref{section:mainresults}, we state our main results \autoref{importantprop}--\autoref{thm:exponentialconvergence}. Finally, in Section~\ref{sec:proofs}, we provide the detailed proofs.

\section{Preliminaries}\label{sec:preli}
In what follows, we introduce some further preliminaries needed to fully define the ambient spaces where our stochastic equations taking place. Let us first recall some useful spaces regarding probability measures \cite{AFMS,panaretos2020,Optimal-Transport-Santambrogio}.
\subsection{Probability measures}
Let $(X,d_{X})$ be a metric space, we denote by $\M(X)$ the space of signed Borel measures with finite total variation, by $\M_{+}(X)$ and $\cP(X)$ the convex subsets of nonnegative measures and probability measures, respectively. For $\sigma\in \M(X)$, $\vert \sigma\vert\in \M_{+}(X)$ denotes the total variation measure of $\sigma$. We shall also use the notation $\M_{0}(X)$ for the subset of measures with zero mean. Given a metric space $(X,d_{X})$, we consider the Lipschitz space 
\begin{align*}
{\rm Lip}(X,d_{X}):=\left\{\phi: X\rightarrow \R\vert \exists\hskip 0,1cm L>0: \forall x,y\in X\hskip 0,1cm \vert \phi(x)-\phi(y)\vert \leq L d_{X}(x,y)\right\}.
\end{align*}
For a continuous function $\phi\in C(X)$ we denote by
\begin{align*}
{\rm Lip}(\phi):=\sup_{\stackrel{x,y\in Y}{x\neq y}}\frac{\vert \phi(x)-\phi(y) \vert}{d_{X}(x,y)}   \end{align*}
its Lipschitz constant. In a complete and separable metric space $(X,d_{X})$, we shall use the Kantorovich-Rubinstein distance $\W_{1}(\cdot,\cdot)$ in the class $\cP(X)$. For $\mu_{1},\mu_{2}\in \M_{1}(X)$, the $1$-Wasserstein distance $\W_{1}(\mu_{1},\mu_{2})$ is defined by 

\begin{align*}
\W_{1}(\mu_{1},\mu_{2}):=\inf\left\{\left. \int_{X\times X}d_{X}(x_{1},x_{2}) \gamma(dx_{1},dx_{2})\right\vert 
 \gamma\in \Gamma(\mu_{1},\mu_{2})\right\}   
\end{align*}
where $\Gamma(\mu_{1},\mu_{2})$ is the set of admissible coupling between $\mu_{1}$ and $\mu_{2}$. It is worth to recall that due to Kantorovich duality, one can also consider the following definition 

\begin{align*}
\W_{1}(\mu_{1},\mu_{2}):=\sup\left\{ \int_{X}\phi{\rm d}(\mu_{1}-\mu_{2})\Big| \hskip 0,1cm \phi \in {\rm Lip}(X,d_{X}), {\rm Lip}(\phi)\leq 1 \right\}.  
\end{align*}

Notice that $\W_{1}(\mu_{1},\nu_{1})$ is finite if $\mu_{1},\nu_{1}$ belong to the space

\begin{align}\label{spaceP1}
\cP_{1}(X):=\left\{ \mu\in \cP(X)\Big| \int_{X}d_{X}(x,\overline{x}){\rm d}\mu(x)<\infty \hskip 0,2cm \text{for some $\overline{x}\in X$}\right\}.
\end{align}

Note that $(\cP_{1}(X),\W_{1})$  is complete if $(X,d_{X})$ is complete. Moreover, by \cite[Theorem 2.2.1]{panaretos2020} the following holds true: a sequence $(\mu_{n})\subset \cP_{1}(X)$ converges to $\mu\in \cP_{1}(X)$ with respect to the Wasserstein distance $\W_{1}$ if and only if,
for all $\phi\in {\rm Lip}(X,d_{X})$,

\begin{align*}
&\int_{X}\phi {\rm d}\mu_{n} \overset{n\rightarrow \infty}{\longrightarrow}  \int_{X} \phi{\rm d}\mu, \qquad  \int_{X}{\rm d}_{X}(\cdot,\bar{x}){\rm d}\mu_{n}\overset{n\rightarrow \infty}{\longrightarrow}  \int_{X}d_{X}(\cdot,\bar{x}){\rm d}\mu.  
\end{align*}
Let $T>0$. In what follows, we denote by $C([0,T],X)$ the space of continuous curves $\y:[0,T]\rightarrow X$ endowed with the metric 

\begin{align}
\de_{T,X}(\y_{1},\y_{2})\coloneqq\sup_{t\in [0,T]}\de_{X}(\y_{1}(t),\y_{2}(t)), \hskip 0,2cm \text{for all $\y_{1},\y_{2}\in C([0,T],X)$.}
\end{align}

\subsection{Our setting}\label{subsec2}
In what follows, we are interested in studying SDEs on $\Y$, where $\Y\coloneqq \Tt^{d}\times \cP(U)$. The state space of our system is given by pairs $(x,\lambda)=y\in\Y$. Here $x\in\Tt^{d}$ denotes the spatial component of an agent, whereas the element $\lambda\in \cP(U)$ denotes a probability distribution over the space $U$, which we assume to be a compact metric space. It can be interpreted as a space of strategies \cite{AFMS,MS2020}. 
Let $(U,d_{U})$ be a compact metric space, and consider the Arens--Ells space $\F(U)$ defined as

 \begin{align}\label{Arens-Ells}
 \F(U)\coloneqq\overline{{\rm span}(\cP(U))}^{{\norma{\cdot}}_{{\rm BL}}}   
 \end{align}
as the closure in the dual space $({\rm Lip}(U,d_{U}))^{\ast}$ with respect to the dual norm

\begin{align*}
{\norma{\ell}}_{{\rm BL}}:=\sup\left\{\nm{\ell,\phi}:\phi\in {\rm Lip}(U, d_{U}), \norma{\phi}_{{\rm Lip}}\leq 1\right\},
\end{align*}
where 

\begin{align}\label{normLip}
\norma{\phi}_{{\rm Lip}}\coloneqq \norma{\phi}_{\infty}+{\rm Lip}(\phi)    
\end{align}
and ${\rm Lip}(\phi)$ is the Lipschitz constant of $\phi$. Introduced in \cite{Arens-Eells}, the space $\F(U)$ is a separable Banach space containing $\M(U)$. Furthermore,  for a measure $\nu\in \M_{0}$, the $\norma{\cdot}_{{\rm BL}}$-norm is equivalent to the norm induced by the dual formulation of the $1$-Wasserstein distance. Furthermore,  we have by Kantorovich duality that

\begin{align}\label{disutile}
   \norma{\mu_{1}-\mu_{2}}_{{\rm BL}}\leq \W_{1}(\mu_{1},\mu_{2})\leq (1+ D_{U})\norma{\mu_{1}-\mu_{2}}_{{\rm BL}},
\end{align}
where 

\begin{align}\label{costantDu}
D_{U}\coloneqq \min_{x_{0}\in U}\max_{x_{1}\in U}\de_{U}(x_{0},x_{1})\le {\rm diam}(U), 
\end{align}
and ${\rm diam}(U)$ denotes the diameter of $U$. In the next we need to consider the space $\cS\coloneqq\Tt^{d}\times \F(U)$ endowed with the norm $\norma{y}_{\cS}=\norma{(x,\sigma)}_{\cS}:=\vert x\vert + \norma{\sigma}_{{\rm BL}}$, which is a separable Banach space. For a given $r>0$, we denote by $B_{r}$ the closed ball of radius $r$ in $\Tt^{d}$ and by $B_{r}^{\Y}$ the ball of radius $r$ in $\Y$, that is, $B_{r}^{\Y}\coloneqq \{y\in\Y:\norma{y}_{\Y}\leq R\}$. We point out that since $U$ is assumed compact, by \cite[Corollary 2.2.5]{panaretos2020} $\cP(U)$ is compact and thus $\Y$ is a locally compact space. Therefore $B_{r}^{\Y}$ is a compact set. Since $\Y\subset \cS$, for any $\Psi\in \cP(\Y)$ we may define
\begin{align*}
m_{p}(\Psi)\coloneqq \int_{\Y}\norma{y}_{\cS}^{p}\de \Psi,
\end{align*}
for $1\leq p<+\infty$. In particular, for $p=1$, one gets 
\begin{align*}
\cP_{1}(\Y)=\{\Psi\in \cP(\Y): m_{1}(\Psi)<+\infty\},
\end{align*}
where $\cP_{1}(\Y)$ as introduced in \eqref{spaceP1}.  

\subsection{The stochastic integral}
In this part, we suppose that $X$ is a real separable Banach space, $X^{\ast}$ its dual. A random variable $\xi:\Omega\rightarrow X$ is called a Gaussian random variable if $\nm{x^{\ast},\xi}: \Omega\rightarrow \Tt$ is a real-valued Gaussian random variable for all $x^{\ast}\in X^{\ast}$. Then, its distribution  is uniquely determined by its mean $\E(\xi)$ and its covariance operator $Q:X^{\ast}\rightarrow X$, $\nm{Q x^{\ast},y^{\ast}}\coloneqq\E(\nm{\xi-\E(\xi),x^{\ast}}\nm{\xi-\E(\xi),y^{\ast}})$ which is a symmetric and positive linear operator. Let us recall that by \autoref{characterization:gaussians}, symmetric and positive linear operators which are covariance operators of Gaussian measures, are then referred to as Gaussian covariances.

\begin{defn}\label{whitenoise}
A family of random elements $(W(t))_{t\in[0,T]}$, $W(t):\Omega\rightarrow X$, is called a Wiener process if
\begin{itemize}
\item $W(0)=0$ almost surely (a.s.);
\item $W(t_{i+1})-W(t_{i}) (i=0,1,\ldots, n-1)$ are independent random elements for every $0\leq t_{0}<t_{1}<\cdots t_{n}\leq T$;
\item for every $t\in [0,T]$, $W(t)$ is a Gaussian random element with mean zero and covariance operator $tQ$, where $Q:X^{\ast}\rightarrow X$ is a fixed Gaussian covariance.
\end{itemize}
\end{defn}
In what follows, we assume that
\begin{align}\label{hyp1}
c_{Q}\coloneqq \sup_{\norma{x^{\ast}}\leq 1}\nm{Q x^{\ast},x^{\ast}}<\infty.
\end{align}
This condition is related to the well-posedness of the definition of stochastic integral in Banach spaces. In the next, our formulation of the stochastic dynamics follows the measure-valued approach introduced in \cite{mamporia2004}. Let us first recall some preliminary concepts.
\begin{defn}[{\cite[Definition 1]{mamporia2004}}]
Let $(\Omega,\B,\mathbb{P})$ be a probability space endowed of a complete filtration $(\F_{t})_{t\in [0,T]}$. A function $\phi:[0,T]\times \Omega \rightarrow X^{\ast}$ is called non-anticipating with respect to $(\F_{t})_{t\in [0,T]}$ if the function $(t,\Omega)\rightarrow \nm{\phi(t,\omega),x}$ from $([0,T]\times \Omega, \B([0,T])\times \B)$ into $(\Tt,\B(\Tt))$ is measurable for all $x\in X$, and the function $\omega\rightarrow \nm{\phi(t,\omega),x}$ from $(\Omega,\B)$ into $(\Tt,\B(\Tt))$ is $\F_{t}$-measurable for all $t\in [0,T]$ and $x\in X$.
\end{defn}
The concept of a non-anticipating function with respect to a complete filtration $(\F_{t})_{t\in [0,T]}$ can be generalized to functions of the form $\phi:[0,T]\times \Omega\rightarrow L(X,X)$, where $L(X,X)$ denotes the space of linear bounded operators from $X$ to $X$. In the next, we set $\H_{Q}(L(X,X))$ the class of all non-anticipating functions $\phi:[0,T]\times \Omega\rightarrow L(X,X)$ such that

\begin{align}\label{nonanti}
\int_{0}^{T}\int_{\Omega}\nm{\phi(t,\omega)Q \phi^{\ast}(t,\omega)x^{\ast},x^{\ast}}\de t\de \P <\infty \hskip 0,2cm \text{for all $x^{\ast}\in X^{\ast}$.}
\end{align}
Here we use $\phi^{\ast}$ to denote the dual map of $\phi$. We now recall what a stochastic integral of a non-anticipating function is.

\begin{defn} \label{defsto}
Given $\phi\in \H_{Q}(L(X,X))$, and $t\in [0,T]$ we consider $T_{\phi}^{t}: X^{\ast}\rightarrow L^{2}(\Omega,\B,\cP)$, $T_{\phi}^{t}x^{\ast}\coloneqq \int_{0}^{t}\nm{\phi^{\ast}(x^{\ast}),\de W(s)}$. We call  $T_{\phi}^{t}$ a generalized stochastic integral of the operator-valued random function $\phi$ with respect to $(W(s))_{s\in [0,t]}$ if $T_{\phi}^{t}$ defines a bounded linear operator. Furthermore, we say that $\xi:\Omega \rightarrow X$ is the stochastic integral of $\phi$ (if such an element exists) if $\nm{\xi,x^{\ast}}=T_{\phi}^{t}x^{\ast}$ for all $x\in X^{\ast}$.
\end{defn}
Notice that Definition \ref{defsto} is meaninful when $\sup_{\norma{x^{\ast}}\leq 1}\norma{T_{\phi}^{t}x^{\ast}}^{2}<\infty$. In fact, we have by an abstract It\"{o} isometry that for each $\phi\in \H_{Q}(L(X,X))$, and $t\in [0,T]$

\begin{align*}
\sup_{\norma{x^{\ast}}\leq 1}\norma{T_{\phi}^{t}x^{\ast}}^{2}&=\sup_{\norma{x^{\ast}}\leq 1}\int_{\Omega}\left(\int_{0}^{t}\phi^{\ast}x^{\ast}\de W(s)\right)^{2}\de \P\\
&=\sup_{\norma{x^{\ast}}\leq 1}\int_{\Omega}\int_{0}^{t}\nm{Q\phi^{\ast}x^{\ast},\phi^{\ast}x^{\ast}}\de s \de\P\\
&=\sup_{\norma{x^{\ast}}\leq 1}\int_{\Omega}\int_{0}^{t}\nm{\phi Q\phi^{\ast}x^{\ast},x^{\ast}}\de s \de\P.
\end{align*}
Because $\phi \in \H_{Q}(L(X,X))$ we then conclude that the last quantity is finite, and we are done. As pointed out in \cite{mamporia2004}, for all $\phi \in \H_{Q}(L(X,X))$ the generalized stochastic integral $T_{\phi}^{t}x^{\ast}$ exists, but the stochastic integral $\xi$ does not always exist. In what follows, if $\xi$ exists, we write $\xi=\int_{0}^{t}\phi \de W_{t}^{Q}$. 
\section{Stochastic dynamics with projected $Q$–noise}
In this part, we introduce the model that we aim to analyze. We need the following space:    

\begin{defn}
We define
\begin{align}
T(U)\coloneqq \{h\in \F(U):\nm{h,1}=0\}. 
\end{align}
\end{defn}
In particular, notice that $T(U)$ is a vectorial space, and thus with the induced norm of $\F(U)$, it has the structure of a Banach space. Furthermore, we define its projector as follows:

\begin{defn}
Let $\Pi_{\lambda}:\F(U)\to T(U)$ be the bounded linear projection
defined by
\begin{align}
\Pi_{\lambda}(h)\coloneqq  h - \nm{h,1}\lambda,    
\end{align}
\end{defn}
In this definition, we substract the component parallel to $\lambda$ to enforce zero total mass:
\begin{align}
\nm{\Pi_\lambda(h),1} = 0.
\end{align}
In what follows, we consider the system of SDEs

\begin{equation}\label{eq:stochastic-system-Q}
\begin{cases}
&\de X^{i}(t)= H(X^i(t),\lambda^i(t),\mu^{N}(t))\,\de t + \sqrt{2\sigma_x}\,\de B^i(t),\\
&\de\lambda^{i}(t)= \R(X^i(t),\lambda^i(t),\mu^{N}(t))\,\de t+
\sqrt{2\sigma_\lambda}\,\Pi_{\lambda^i(t)}\,\de W^{Q,i}(t),
\end{cases}
\end{equation}
where $i=1,\ldots, N$, $B^{i}$ is a standard Brownian motion in $\mathbb{R}^d$, $W^{Q,i}$ is the generalized noise on $\F(U)$, $\Pi_{\lambda^i(t)}$ ensures that the stochastic increment of $\lambda^i(t)$ lies in $T_{\lambda^{i}(t)}(U)$, $\sigma_x,\sigma_\lambda>0$ are the diffusion coefficients, and $\mu^{N}(t)$ is the empirical measure defined as
      \begin{align}\label{empiricalmeasure}
       \mu^{N}(t)\coloneqq \frac{1}{N}\sum_{i=1}^{N}\delta_{(X^{i}(t),\lambda^{i}(t))}.   
      \end{align}
Here, we suppose that $\{B^{i}, W^{Q,i}\}_{i=1}^{N}$ are independent random variables. An important fact is that the second equation is intended in the following sense:

\begin{defn}
Suppose that $(\Omega,\B,\P)$ is a complete probability space endowed with the $\sigma$-algebra generated by $(B^{i}(t)\otimes W^{Q,i}(t))_{t\in[0,T],i=1\ldots,N}$ and that is still denoted as $(\F_{t})_{t\in[0,T]}$. We say that an $\F(U)$-valued predictable process $\lambda^{i}(t)$, $t\in[0,T]$, is a weak solution to 

\begin{align}
&\de\lambda^{i}(t)= \R(X^i(t),\lambda^i(t),\mu^{N}(t))\,\de t+
\sqrt{2\sigma_\lambda}\,\Pi_{\lambda^i(t)}\,\de W^{Q,i}(t)   
\end{align}
if  for all $x^{\ast}\in (\F(U))^{\ast}$

\begin{align}\label{integ1}
\nm{\lambda^{i}(t),x^{\ast}}=\nm{\lambda^{i}(0),x^{\ast}}+ \int_{0}^{t}\nm{\R(X^i(s),\lambda^i(s),\mu^{N}(s)),x^{\ast}}\de s+ \sqrt{2\sigma}T_{\I_{i}}^{t}x^{\ast}
\end{align}
where $\I_{i}\in \H_{Q}(L(\F(U),\F(U)))$ denotes the projector operator $\Pi_{\lambda^{i}}$. Moreover, we say that $(\lambda^{i}(t))_{t\in[0,T]}$ is a strong solution if 
\begin{align}\label{integ2}
\lambda^{i}(t)=\lambda^{i}(0)+ \int_{0}^{t}\R(X^i(s),\lambda^i(s),\mu^{N}(s))\de s+ \sqrt{2\sigma}\int_{0}^{t}\Pi_{\lambda^{i}(s)}\de s \hskip 0,4cm \text{$\P$-a.s..}
\end{align}
\end{defn}
\begin{rem}
Notice that thanks to \eqref{hyp1} one has that for each $t\in [0,T]$, and any $i=1,\ldots, N,$
\begin{align*}
\sup_{\norma{x^{\ast}}\leq 1}\norma{T_{\I_{i}}^{t}x^{\ast}}^{2}&\leq \sup_{\norma{x^{\ast}}\leq 1}\int_{\Omega}\int_{0}^{T}\nm{Q x^{\ast},x^{\ast}}\de t \de\P\\
&=T\sup_{\norma{x^{\ast}}\leq 1}\nm{Q x^{\ast},x^{\ast}}<\infty.
\end{align*}
\end{rem}
To give sense to the dynamics of $\lambda^{i}(t)$ in the form of \eqref{integ2}, we aim to restrict ourselves to operators satisfying the following definition.

\begin{defn}
 We say that a positive symmetric linear operator $A:(\F(U))^{\ast}\rightarrow \F(U)$ belongs to $\mathscr{R}_{2}(\F(U))$ if there exists a positive constant $c$ such that for all $n\in\N$, and $x_{1}^{\ast},\ldots, x_{n}^{\ast}$ belonging to $(\F(U))^{\ast}$

 \begin{align}\label{condition_dual}
 \left(\sum_{i=1}^{n}\norma{A x_{i}^{\ast}}_{{\rm BL}}^{2}\right)^{\frac{1}{2}}\leq  c\sup_{\nm {Ax^{\ast},x^{\ast}}\leq 1}\left(\sum_{i=1}^{n}\left\vert \nm{A x_{i}^{\ast},x^{\ast}} \right\vert^{2}\right)^{\frac{1}{2}}.   
 \end{align}
\end{defn}
We have the following:
\begin{thm}[{\cite[Theorem 1]{mamporia2004}}]\label{thm:mamporia}
Let $\phi\in \H_{Q}(L((\F(U),\F(U))))$, and $L_{\phi}: (\F(U))^{\ast}\rightarrow \F(U)$, $x^{\ast}\mapsto \int_{0}^{t}\int_{\Omega}\phi(s,w)Q\phi^{\ast}(s,\omega)x^{\ast}\de s\P(\omega)$ belongs to $\mathscr{R}_{2}(\F(U))$. Then there exists the stochastic integral $\int_{0}^{t}\phi \de W_{s}^{Q}$, and
\begin{align}
  \E\norma{\int_{0}^{t}\phi \de W_{s}^{Q}}_{{\rm BL}}^{2}<+\infty.  
\end{align}
Furthermore, the process $\int_{0}^{t}\phi \de W_{s}^{Q}$ has a.s. continuous sample paths.
\end{thm}
\begin{proof}
The proof of this result can be found in \cite[Theorem 1]{mamporia2004} for general Banach spaces with interval time $[0,1]$. The same proof holds true for every interval $[0,T]$, with $0<T<+\infty$.    
\end{proof}
\begin{rem}[On the dual structure of $\F(U)$]
Let $x_0\in U$ be fixed and denote by $\mathrm{Lip}_0(U)$ the space of real-valued Lipschitz functions on $U$ vanishing at $x_0$, namely
\begin{align}
\mathrm{Lip}_0(U)\coloneqq \{ f\in \mathrm{Lip}(U): f(x_0)=0\},    
\end{align}
endowed with the Lipschitz norm. It is well known that $\F(U)$ is the canonical predual of $\mathrm{Lip}_0(U)$. More precisely, one has the isometric identification
\begin{align}
(\F(U))^{*} \simeq \mathrm{Lip}_0(U).    
\end{align}
We refer for instance to \cite[Theorem 2.11]{fodoroptimal} for a precise statement. In the present work, since we do not fix explicitly a base point $x_0$, we keep the notation $(\F(U))^{*}$ for the dual space, with the understanding that it can be identified with a suitable Lipschitz space. This dual characterization justifies the formulation of the cylindrical noise through an operator $Q:(\F(U))^{*}\to\F(U)$ and explains the compatibility between the $\BL$ norm and the Wasserstein-$1$ distance used throughout the paper. 
\end{rem}
\begin{ex}
As pointed out in \cite[Example 2.9]{fodoroptimal}, by restricting ourselves to $U=\{0,1\}$, we have that $\F(U)\cong \mathbb{R}$, and thus $T(U)=\{0\}$. Therefore, equation \eqref{eq:stochastic-system-Q} is reduced to a particular case of the model studied in \cite{d2025large}. 
\end{ex}
Before of giving a nontrivial example, let us first recall the notion of propagation of chaos, and introduce our setting. Our interest relies about the propagation of chaos of the system of SDEs \eqref{eq:stochastic-system-Q}. To this aim we consider the following:
\begin{defn}[Propagation of chaos, \cite{sznitman1991topics}]\label{defn:propchaos}
A sequence of random vectors $\by^N=(Y^1,\dots,Y^N)$ in $\Y^N$ is said to be \emph{chaotic with law $\mu\in\cP(\Y)$} if for every $k\in\mathbb N$ fixed, the $k$-marginals converge in law:
\begin{align}
\mathcal L(Y^{1},\dots,Y^{k}) \xrightarrow[N\to\infty]{\text{law}} \mu^{\otimes k}.    
\end{align}
\end{defn}
To analyze the asymptotic behavior of \eqref{eq:stochastic-system-Q} we consider the following assumptions.
\subsection{Assumptions}\label{sub:assumptions}
We assume that given $\Psi\in \cP(\Y)$ the velocity field $H(\cdot,\cdot,\Psi):\Y\rightarrow \Tt^{d}$ satisfies the following structural assumptions:
\begin{ass}\label{A1}
For every $r>0$, for every $\Psi\in \cP(B_{r}^{\Y})$, $H(\cdot,\cdot,\Psi)\in {\rm Lip}(B_{r}^{\Y};\Tt^{d})$ uniformly with respect to $\Psi$, that is, there exists a positive constant $C_{H,r}$ such that
\begin{align*}
\norma{H(y^{1},\Psi)-H(y^{2},\Psi)}_{2}\leq C_{H,r}\norma{y^{1}-y^{2}}_{\cS};
\end{align*}
\end{ass}
\begin{ass}\label{A2}
For every $r>0$, for every $\Psi\in\cP(B_{r}^{\Y})$, there exists a positive constant $C_{H,r}$ such that for every $y\in B_{r}^{\Y}$ and for every $\Psi^{1},\Psi^{2}\in \cP(B_{r}^{\Y})$
\begin{align*}
\norma{H(y,\Psi^{1})-H(y,\Psi^{2})}_{2} \leq C_{H,R}\W_{1}(\Psi^{1},\Psi^{2});
\end{align*}
\end{ass}
\begin{ass}\label{A3}
 There exists $M_{H}>0$ such that for every $y\in\Y$ and for every $\Psi\in\cP_{1}(\Y)$ there holds 
\begin{align*}
\norma{H(y,\Psi)}_{2}\leq M_{H}\left(1+\norma{y}_{\cS}+ m_{1}(\Psi)\right). 
\end{align*} 
\end{ass}
In what follows,  we state the assumptions on the operator $\R$. 
\begin{ass}\label{A4}
 For $\Psi\in \cP_{1}(\Y)$, we suppose that the operator $\R(\cdot,\Psi):\Y\rightarrow \F(U)$ satisfies that for every $(y,\Psi)\in \Y\times \cP_{1}(\Y)$ constants belong to the kernel of $\R(y,\Psi)$, that is
\begin{align*}
\nm{\R(y,\Psi),1}_{\F(U),{\rm Lip}(U)}=0.
\end{align*}   
\end{ass}
\begin{ass}\label{A5}
For every $(y,\Psi)\in \Y\times \cP_{1}(\Y)$, there exists a positive constant $M_{\R}$ such that
\begin{align*}
\norma{\R(y,\Psi)}_{{\rm BL}}\leq M_{\R}(1+\norma{x}_{2}+m_{1}(\Psi)), \hskip0,1cm\text{where $y=(x,\lambda)$.}
\end{align*}    
\end{ass}
\begin{ass}\label{A6}
For every $r>0$ there exists  a positive constant $L_{\R,r}$ such that for every $(y^{1},\Psi^{1}),(y^{2},\Psi^{2})\in B_{r}^{\Y}\times \cP(B_{r}^{\Y})$,
\begin{align*}
\norma{\R(y^{1},\Psi^{1})-\R(y^{2},\Psi^{2})}_{{\rm BL}}\leq L_{\R,r}\left(\norma{y^{1}-y^{2}}_{\cS}+ \W_{1}(\Psi^{1},\Psi^{2})\right);
\end{align*}    
\end{ass}
\begin{ass}\label{A7}
For every $r>0$ there exists $\delta_{r}>0$ such that for every $(y,\Psi)\in B_{r}^{\Y}\times \cP_{1}(\Y)$ we have
\begin{align*}
\R(y,\Psi)+\delta_{r}\lambda\geq 0,  \hskip0,1cm\text{where $y=(x,\lambda)$}.
\end{align*} 
\end{ass}
\begin{ass}\label{A8}
We suppose that $Q:(\F(U))^{\ast}\rightarrow \F(U)$ is a symmetric and positive operator such that \eqref{hyp1} holds true, and for each $\lambda\in \F(U)$, and each non-negative function $\phi\in C_{b}(U)$

\begin{align}
 \nm{Q\Pi^{\ast}_{\lambda}\phi,\Pi^{\ast}_{\lambda}\phi}\leq (\nm{\lambda,\phi})^{+}.   
\end{align}
where $^{+}$ denotes the positive part.
\end{ass}
Let us comment the previous assumptions.
\begin{rem}
Assumptions \autoref{A1}-\autoref{A2} and \autoref{A6} reflect the local Lipschitz property of the fields $H$ and $\R$. These hypotheses, combined with the linear growth conditions \autoref{A3} and \autoref{A5}, ensure that the dynamics remain confined to a ball $B_{r'}^{\Y^{N}}$ of some positive radius $r'$, which depends on the time horizon $T>0$ and the constants appearing in our assumptions. Assumption \autoref{A4} expresses that perturbations of a probability measure occur through its tangent space. This condition, together with \autoref{A7}-\autoref{A8}, guarantees the preservation of positivity of $\lambda(t)$ when perturbed by the generalized noise $W^{Q}(t)$ and by the field $\R$. We also note that the noise $W^{Q}(t)$ does not require $Q$ to be compact; in fact, our setting allows $W^{Q}(t)$ to be cylindrical.
\end{rem}
Assumptions \autoref{A1}--\autoref{A8} ensure well-posedness of the particle system and of the McKean--Vlasov equation on any finite time interval $[0,T]$, together with propagation of chaos on finite horizons. In order to extend these results to infinite time horizons, we impose the gradient structure and uniform convexity assumptions \autoref{A9}--\autoref{A12}. By restricting ourselves to these hypotheses, we guarantee global dissipativity of the dynamics and exponential stability of the McKean--Vlasov flow.

\begin{ass}\label{A9}
We assume that both drift terms are of gradient type:
\begin{align}
H(x,\lambda,\mu) &= - \nabla_x F(x,\lambda,\mu), \\
\R(x,\lambda,\mu) &= - D_\lambda G(x,\lambda,\mu),
\end{align}
for suitable potentials $F$ and $G$.    
\end{ass}
\begin{ass}[Uniform convexity in $x$]\label{A10}
There exists $m_x > 0$ such that for all $x,x' \in \Tt^d$, $\lambda \in \cP(U)$, and $\mu \in \cP_2(\Y)$,
\begin{align}
\big\langle 
\nabla_x F(x,\lambda,\mu) - \nabla_x F(x',\lambda,\mu),
\, x-x'
\big\rangle
\ge
m_x \|x-x'\|_{2}^2.
\end{align}
\end{ass}
\begin{ass}[Uniform convexity in $\lambda$]\label{A11}
There exists $m_\lambda > 0$ such that for all $\lambda,\lambda' \in \cP(U)$, $x \in \Tt^d$, and $\mu \in \cP_{1}(\Y)$,
\begin{align}
\big\langle 
D_\lambda G(x,\lambda,\mu) - D_\lambda G(x,\lambda',\mu),
\, \lambda-\lambda'
\big\rangle
\ge
m_\lambda \|\lambda-\lambda'\|_{\BL}^2,
\end{align}    
where we have denoted with $D_{\lambda}$ the Fréchet derivative of $G$ along $\lambda$ (see \autoref{sec:frechet}). 
\end{ass}

\begin{ass}\label{A12}
There exists $L_\mu > 0$ such that for all $(x,\lambda)\in \Y$ and $\mu,\nu \in \cP_{1}(\Y)$,
\begin{align}
|\nabla_x F(x,\lambda,\mu) - \nabla_x F(x,\lambda,\nu)|
&\le
L_\mu \W_{1}(\mu,\nu),\\
\|D_\lambda G(x,\lambda,\mu) - D_\lambda G(x,\lambda,\nu)\|
&\le
L_\mu \W_{1}(\mu,\nu).
\end{align}
\end{ass}
\begin{rem}
We emphasize that $D_{\lambda} G$ denotes the Fréchet derivative of $G$ with respect to $\lambda$ in the Banach space $(\F(U),\|\cdot\|_{\BL})$. 
In particular, differentiability is understood in the $\BL$ topology rather than in the Wasserstein sense introduced in \cite{AGS}.

Although the Lipschitz assumptions \autoref{A2}, \autoref{A6}, and \autoref{A12} are formulated in terms of the Wasserstein distance $\W_1$, this is consistent with our functional framework. 
Indeed, on a compact set $U$, the Wasserstein--1 distance is controlled by the bounded--Lipschitz norm as stated in \eqref{disutile} up to a multiplicative constant depending only on the diameter of $U$. 
Therefore, Lipschitz continuity with respect to $\W_1$ is compatible with differentiability in the $\BL$ norm.

This choice is natural in our setting, since the measure-valued component evolves in the Arens--Eells space and the projected cylindrical noise is defined in the dual of $\F(U)$. 
For this reason, the Fréchet derivative in $\BL$ provides the appropriate notion of gradient for the analysis developed in this work.
\end{rem}

\begin{ex}
Let us consider a rank-one operator $Q: \F(U) \to \F(U)$ as:
\begin{align}\label{eq:Q_rank_one}
Q(u)\coloneqq \varepsilon\nm{ u, u_0} \, u_0,
\end{align}
where $\kappa\in (0, 1)$ is a fixed parameter and $u_0 \in \F(U)$ is a fixed element. Notice that $Q$ is bounded, linear, and semi definite-positive. Let us set $U = [0,1]$, and consider $\lambda \in \F(U)$, and $\phi \in C_b(U)$ be a non-negative, and continuous bounded function. In what follows, suppose that
\begin{align}\label{controlled}
\left|\left\langle \phi, u_0 - \langle u_0, 1 \rangle \lambda \right\rangle\right| \leq \epsilon \sqrt{(\langle \lambda, \phi \rangle)^+}.
\end{align}
where $\epsilon > 0$ is a fixed constant. By assuming that $\kappa \epsilon^2< 1$, then
\begin{align*}
\nm{Q(\Pi_\lambda^* \phi), \Pi_\lambda^{\ast} \phi}
&= \kappa\epsilon^{2} \nm{ \phi, u_0 - \langle u_0, 1 \rangle \lambda}^2 \\
&<\nm{\lambda, \phi}^{+},
\end{align*}
and \autoref{A8} holds true. In the context of machine learning, operators of the form as in \eqref{eq:Q_rank_one} are particularly valuable since naturally emerge in low-rank approximations of high-dimensional covariance matrices and in the analysis of attention  mechanisms in neural networks~\cite{Berlinet2004}. The controlled coupling condition \eqref{controlled} ensures that perturbations remain sub-linear in the mean of test functions $\phi$, providing automatic regularization of the bilinear form. This property has direct applications in robust statistical estimation, where noise structure must be carefully controlled to prevent destabilization of estimators under model misspecification \cite{bartolucci2024neural}. The constraint $\kappa\epsilon^2 < 1$ plays a critical role in stabilizing the noise interactions. This bound guarantees that the self-reinforcing feedback through the direction $u_0$ does not amplify deviations unboundedly~\cite{kolesnikov2026general,li2026optimal}.
\end{ex}

The following result will be used to prove that $\lambda(t)$ is a probability measure $a.s..$
\begin{thm}[{\cite[Thm. 1.2, Ch. VI]{revuz2013continuous}}]\label{thm:Tanaka}
Let $(Z_t)_{t\ge0}$ be a real-valued continuous semimartingale. For any $a\in\Tt$, there exists a stochastic process $L_t^a(Z)$ called the local time of $Z$ in $a$ such that
\begin{align}
(Z_t-a)^-
=(Z_0-a)^- -\displaystyle\int_0^t \mathbf 1_{\{Z_s<a\}}\,\de Z_s
+\frac{1}{2}\,L_t^a(Z),
\qquad t\ge0,
\end{align}
where $x^-=\max\{-x,0\}$ and $L_t^a(Z)$ denotes the local time of $Z$ at level $a$.
\end{thm}
\begin{rem}
The local time $L_t^a(Z)$ is a continuous, nondecreasing adapted process,
which admits the representation
\begin{align}\label{localtime}
L_t^a(Z)=\lim_{\varepsilon\downarrow0}\frac{1}{\varepsilon}
\int_0^t \mathbf 1_{\{|Z_s-a|<\varepsilon\}}
\,\de\nm{Z}_s,
\end{align}
where $\nm{Z}_s$ denotes the quadratic variation of $Z$.
\end{rem}

\begin{lem}\label{lem:positivemass}
Suppose that \autoref{A4}, and \autoref{A7}-\autoref{A8} hold true. Let $(\lambda(t))_{t\ge0}$ be a solution to
\begin{align}\label{eqperlambda}
\de\lambda(t)
= \R(y(t),\Psi)\,\de t+ \sqrt{2\sigma_\lambda}\,\Pi_{\lambda(t)}\,\de W^Q(t),
\end{align}
with initial condition $\lambda_0\in\cP(U)$, and $\Psi\in \cP(\Y)$. Then

\begin{align}
\lambda(t)\in\cP(U)\qquad\text{for all }t\ge0,\ \text{a.s..}    
\end{align}
\end{lem}
\begin{proof}
We split the proof into conservation of mass and positivity. Let us first prove the conservation of mass. Notice that
\begin{align}
\de\nm{\lambda(t),1}&= \nm{\R(y(t),\Psi),1}\de t+ \sqrt{2\sigma_\lambda}\,\nm{\Pi_{\lambda(t)}\,\de W^Q(t),1}.
\end{align}
By \autoref{A4} and by the definition of $\Pi_\lambda$,
\begin{align}
 \nm{\R(y,\Psi),1}=0,\hskip 0,2cm \nm{\Pi_\lambda h,1}   
\end{align}
for all $(y,\Psi)$ and $h\in\F(U)$.  
Therefore,
\begin{align}
\de\nm{\lambda(t),1}=0.    
\end{align}

and since $\langle\lambda_0,1\rangle=1$, we obtain
\begin{align}
\langle \lambda_t,1\rangle=1\quad\text{for all }t\ge0.    
\end{align}
Let us now prove positivity. Let $\phi\in C(U)$ with $\phi\ge0$ and define

\begin{align}
 Z_{t}\coloneqq \nm{\lambda(t),\phi}.   
\end{align}
Then $(Z_t)_{t\ge0}$ is a real-valued continuous semimartingale satisfying
\begin{align}
\de Z_{t}=\nm{\R(y(t),\Psi),\phi}\de t+ \sqrt{2\sigma_\lambda}\nm{\Pi_{\lambda(t)}\,\de W^Q(t),\phi}.
\end{align}
We apply Tanaka's formula given in \autoref{thm:Tanaka} to the negative part $Z_t^-=\max\{-Z_t,0\}$:
\begin{align*}
\mathrm d Z_t^-
&= -\mathbf 1_{\{Z_t<0\}}\,\mathrm d Z_t
+ \frac12\,\mathrm d L_t^0(Z),
\end{align*}
where $L_t^0(Z)$ denotes the local time of $Z$ at $0$.  Let us prove that $L_{t}^{0}(Z)$ is zero. Notice that the quadratic variation of $Z$ is given by

\begin{align}
\de \nm{Z}_{t}=2\sigma_{\lambda}\nm{Q\Pi_{\lambda(t)}^{\ast}\phi,\Pi_{\lambda(t)}^{\ast}\phi}\de t.    
\end{align}
Then by \autoref{A8} one gets that

\begin{align}
\de \nm{Z}_{t}\leq 2\sigma_{\lambda}Z_{t}^{+}\de t.  
\end{align}
Then 

\begin{align}
\frac{1}{\varepsilon}
\int_0^t \mathbf 1_{\{|Z_s|<\varepsilon\}}
\,\de\nm{Z}_s &\leq \frac{2\sigma_{\lambda}}{\varepsilon}
\int_0^t \mathbf 1_{\{|Z_s|<\varepsilon\}}
\,Z_{s}^{+}\de s\\
&\leq \frac{2\sigma_{\lambda}}{\varepsilon}\varepsilon \int_{0}^{t}\mathbf 1_{\{|Z_s|<\varepsilon\}}\de s,
\end{align}
and by using the formula \eqref{localtime} we conclude that $L_{t}^{0}(Z)=0$. Now, by \autoref{A7},

\begin{align}
\nm{\R(y(t),\Psi),\phi}\ge -\delta_r \nm{\lambda(t),\phi}= -\delta_r Z_t.
\end{align}
Hence, on the set $\{Z_t<0\}$,
\begin{align}
-\mathbf 1_{\{Z_t<0\}}\nm{\R(y_t,\Psi),\phi}\le \delta_r Z_t^-.    
\end{align}
Taking expectations and using that the stochastic integral has zero mean, we obtain

\begin{align}
\frac{\de}{\de t}\E[Z_t^-]
\le \delta_r\,\E[Z_t^-].  
\end{align}
Since $\lambda_0\in\cP(U)$, we have $Z_0^-=0$, and Gr\"onwall's lemma yields
\begin{align}
\E[Z_t^-]=0\quad\text{for all }t\ge0.    
\end{align}
As $Z_t^-\ge0$, this implies $Z_t^-=0$ almost surely, i.e.
\begin{align}
\langle \lambda(t),\phi\rangle\ge0.   
\end{align}
Since $\phi\ge0$ was arbitrary, $\lambda_t$ is a nonnegative measure. Combining positivity with conservation of mass, we conclude that
$\lambda(t)\in\cP(U)$ for all $t\ge0$, almost surely.
\end{proof}
\begin{prop}\label{prop:R2}
Assume that \eqref{hyp1} holds true. Then for each $\lambda\in \F(U)$ the stochastic integral $\int_{0}^{t}\Pi_{\lambda}\de W^{Q}(s)$ is well defined as an $\F(U)$--valued process, and
\begin{align}
  \E\norma{\int_{0}^{t}\Pi_{\lambda} \de W^{Q}(s)}_{{\rm BL}}^{2}<+\infty.  
\end{align}
Furthermore, the process $\int_{0}^{t}\Pi_{\lambda} \de W_{s}^{Q}$ has a.s. continuous sample paths.
\end{prop}

\begin{proof}
In what follows, we make use of \autoref{thm:mamporia}. In particular, we prove that $L_{\Pi_{\lambda}}: (\F(U))^{\ast}\rightarrow \F(U)$, $x^{\ast}\mapsto \int_{0}^{t}\int_{\Omega}\Pi_{\lambda}Q\Pi_{\lambda}^{\ast}(s,\omega)x^{\ast}\de s\P(\omega)$ belongs to $\mathscr{R}_{2}(\F(U))$. Recall that
\begin{align}
\Pi_\lambda(h)=h-\nm{h,1}\lambda,\qquad h\in\F(U).    
\end{align}
Notice that linearity is immediate. We now prove that $\Pi_\lambda$ is a bounded linear operator on $\F(U)$. Notice by \eqref{normLip}, we conclude that
\begin{align}
\|\Pi_\lambda(\ell)\|_{\rm BL}
\le
(1+\|\lambda\|_{\rm BL})\,\|\ell\|_{\rm BL}.   
\end{align}
Hence, the operator norm of $\Pi_{\lambda}$ satisfies

\begin{align}\label{Pi_bound}
\|\Pi_\lambda\|_{L(\F(U),F(U))}\le 1+\|\lambda\|_{\rm BL}.    
\end{align}
and its dual map $\Pi_\lambda^{\ast}\in L((\F(U))^{\ast},(\F(U))^{\ast})$ is bounded as well. Let us now prove the $\mathscr{R}_2$ condition. Let $x_1^{\ast},\dots,x_n^{\ast}\in(\F(U))^{\ast}$. By definition of the bounded--Lipschitz norm,

\begin{align}
\norma{\Pi_\lambda Q \Pi_\lambda^{\ast}x_i^{\ast}}_{\BL}
=\sup_{\|\phi\|_{\rm Lip}\le 1}
\bigl|
\nm{\Pi_\lambda Q \Pi_\lambda^{\ast}x_i^{\ast},\phi}
\bigr|.   
\end{align}

Using duality and the definition of $\Pi_\lambda^{\ast}$,

\begin{align}
\nm{\Pi_\lambda Q \Pi_\lambda^{\ast}x_i^{\ast},\phi}
=
\nm{Q\Pi_\lambda^{\ast}x_i^{\ast},\Pi_\lambda^{\ast}\phi}.
\end{align}
By \eqref{Pi_bound}, we have $\|\Pi_\lambda^{\ast}\phi\|\le (1+\norma{\lambda}_{\BL})\|\phi\|_{\rm Lip}\le (1+\norma{\lambda}_{\BL})$. Now, using the Cauchy--Schwarz inequality associated with the positive symmetric operator $Q$,
\begin{align}
|\nm{Q u^{\ast},v^{\ast}}|
\le
\sqrt{\nm{Q u^{\ast},u^{\ast}}}\sqrt{\nm{Q v^{\ast},v^{\ast}}},    
\end{align}
together with assumption \eqref{hyp1}, we obtain

\begin{align}
|\nm{Q\Pi_\lambda^{\ast}x_i^{\ast},\Pi_\lambda^{\ast}\phi}|
\le \sqrt{c_Q}\,\sqrt{\nm{Q\Pi_\lambda^{\ast}x_i^{\ast},\Pi_\lambda^{\ast}x_i^{\ast}}}.    
\end{align}
Taking the supremum over $\phi$ with $\|\phi\|_{\rm Lip}\le 1$, we deduce

\begin{align}
\norma{\Pi_\lambda Q \Pi_\lambda^{\ast}x_i^{\ast}}_{\rm BL}
\le
\sqrt{c_Q}\,\sup_{\nm{Q y^{\ast},y^{\ast}}\le 1}
|\nm{Q\Pi_\lambda^{\ast}x_i^{\ast},y^{\ast}}|.    
\end{align}
Squaring and summing over $i=1,\dots,n$ yields

\begin{align}
\sum_{i=1}^n
\norma{\Pi_\lambda Q \Pi_\lambda^{\ast}x_i^{\ast}}_{\rm BL}^2
\le
c_Q
\sup_{\nm{Q y^{\ast},y^{\ast}}\le 1}
\sum_{i=1}^n
|\nm{\Pi_\lambda Q \Pi_\lambda^{\ast}x_i^{\ast},y^{\ast}}|^2.
\end{align}
This is precisely condition \eqref{condition_dual}. Hence
$L_{\Pi_\lambda}\in\mathscr R_2(\F(U))$, and the claim follows from
Theorem~\ref{thm:mamporia}.
\end{proof}
Furthermore, it is possible to prove the well-posedness of the integral form for \eqref{eqperlambda}.
\begin{lem}
Let us suppose that \autoref{A4}-\autoref{A5}, \autoref{A7}-\autoref{A8}, and consider $y(t)=(x(t),\lambda(t))$ be a random variable with values in $\Y$ such that $\E\sup_{r\in [0,T]}\norma{y(r)}_{\Y}^{2}<+\infty$, and where $\lambda(r)$ follows the differential equation \eqref{eqperlambda}. Then the integral form for \eqref{eqperlambda}, namely, 

\begin{align}
\lambda(t)=\lambda(0)+\int_{0}^{t}\R(y(t),\Psi)\de t + \int_{0}^{t}\Pi_{\lambda(s)}\de W^{Q}(s)
\end{align}
makes sense with respect to the seminorm $\E\norma{\cdot}_{\BL}^{2}$.
\end{lem}
\begin{proof}
Let us take $0\leq s<t\leq T$. Notice that 

\begin{align}
\E\norma{\lambda(t)-\lambda(s)}_{\BL}^{2}&\leq 2\E\norma{\int_{s}^{t}\R(y(r),\Psi) \de r}_{\BL}^{2} + 
  \E\norma{\int_{s}^{t}\Pi_{\lambda(r)} \de W^{Q}(r)}_{{\rm BL}}^{2}\\  
 &\leq 2(t-s)\int_{s}^{t}\E\norma{\R(y(r),\Psi)}_{\BL}^{2}\de r+ 2C_{Q}\int_{s}^{t}(1+\E\norma{\lambda(r)}_{\BL}^{2})\de (r)
\end{align}
where in the last inequality we have used Jensen inequality, and \eqref{Pi_bound}. By \autoref{A5} one has for a suitable positive constant $C>0$ that

\begin{align}
\int_{s}^{t}\E\norma{\R(y(r),\Psi)}_{\BL}^{2}\de r&\leq \int_{s}^{t}\E\left(1+\norma{x(r)}_{2}+m_{1}(\Psi)\right)^{2}\de r\\
\leq C(t-s)^{2}
\end{align}
where $y(r)=(x(r),\lambda(r))$, and in the last inequality we have used the assumption about the second moment of $y(\cdot)$, and that $m_{1}(\Psi)<+\infty$. Hence, we can find a positive constant $C'$ such that

\begin{align}
\E\norma{\lambda(t)-\lambda(s)}_{\BL}^{2}\leq C'(t-s),    
\end{align}
and we are done.
\end{proof}
\begin{rem}
For simplicity we have assumed the uniform second moment bound 
\begin{align}
\E\sup_{r\in[0,T]}\|y(r)\|_{\Y}^{2}<+\infty.    
\end{align}
In the case of the interacting particle system \eqref{sistema1}, this estimate is established a posteriori from the linear growth assumptions on $H$ and $\R$. We also point out that, in the absence of the stochastic perturbation in the $\lambda$-equation, the equivalence between the differential and the integral formulations has been discussed in \cite[Subsection 2.2]{AFMS}. In that deterministic setting, the integral formulation admits a natural weak interpretation. The present stochastic framework extends that structure to the case of cylindrical noise.
\end{rem}
\subsection{The vectorial multiagent system}
In what follows, we introduce the  vector-valued variable $\y\coloneqq (y^{1},\ldots,y^{N})\in \Y^{N}\subset \cS^{N}$, which we endow with the norm 
\begin{align}
\norma{\y}_{\cS^{N}}\coloneqq \frac{1}{N}\sum_{i=1}^{N}\norma{y^{i}}_{\cS}.
\end{align}
Noitce that the natural space to study the well posedness of the system described through \eqref{eq:stochastic-system-Q} is $\cS^{N}$. However, notice that $\hat{\Y}^{N}\coloneqq (\Tt^{d})^{N}\times(\cP(U))^{N}$, and $\hat{\cS}^{N}\coloneqq (\Tt^{d})^{N}\times(\F(U))^{N}$ can be also endowed with the same norm $\norma{\y}_{\cS^{N}}$, making them more convenient spaces to analyze such a system. In the next, we consider $\blambda\coloneqq (\lambda^{1},\ldots,\lambda^{N})\in (\cP(U))^{N}$, and $\X\coloneqq (X^{1},\ldots, X^{N})\in (\Tt^{d})^{N}$. For each $\Psi\in \cP_{1}(\Y)$, we consider the map $\vH^{N}(\cdot,\Psi): \hat{\Y}^{N}\rightarrow (\Tt^{d})^{N}$ which is defined through

\begin{align}\label{H_vectorialfield}
\vH^{N}(\X,\blambda,\Psi)\coloneqq (H(X^{1},\lambda^{1},\Psi),\ldots,H(X^{N},\lambda^{N},\Psi)).
\end{align}
Furthemore, we define the map $\bR_{\Psi}^{N}:\hat{\Y}^{N}\rightarrow (\F(U))^{N}$ as
\begin{align*}
\bR_{\Psi}^{N}(\X,\blambda)\coloneqq (\R(X^{1},\lambda^{1},\Psi),\ldots,\R(X^{N},\lambda^{N},\Psi)),
\end{align*}
and we set $\bB(t)\coloneqq (B^{1}(t),\ldots, B^{N}(t))$ as the $d\times N$-valued Brownian motion. Then we write \eqref{eq:stochastic-system-Q} in the compact form

\begin{align}\label{system2}
\begin{dcases}
&\X(t)=\X_{0}+ \int_{0}^{t} \vH^{N}(\X(s),\blambda(s),\mu^{N}(s))\de s+ \sqrt{2\sigma_{x}}\bB(t),\\
&\blambda(t)=\blambda_{0}+ \int_{0}^{t}\bR^{N}(\X(s),\blambda(s),\mu^{N}(s))\de s+\sqrt{2\sigma_{\lambda}}\int_{0}^{t}\Pi_{\blambda(t)}\de W^{\bQ}(s)
\end{dcases}
\end{align}
where
\begin{align}
W^{\bQ}(t)\coloneqq (W^{Q,1}(t),\ldots, W^{Q,N}(t)).   
\end{align}
In what follows, we suppose that ${\bf X}_{0}\in L^{2}(\Omega;(\Tt^{d})^{N})$, and ${\bf \lambda}_{0}\in (\cP_{1}(U))^{N}$ for all $N\in\N$. Since the above is an SDE, we now recall what a strong solution is for \eqref{system2}.

\begin{defn}\label{defstrong}
Suppose that $(\Omega,\B,\P)$ is a complete probability space endowed with the $\sigma$-algebra generated by $(\bB(t)\otimes W^{\bQ}(t))_{t\in[0,T]}$ and that we denote as $({\mathcal F}_{t})_{t\in[0,T]}$. We say that an $\cS^{N}$-valued predictable process $\by(t)=(\X(t),\blambda(t))$, $t\in[0,T]$ is a strong solution of \eqref{system2} if $\by(t)$ satisfies $\P$-a.s. \eqref{system2}, and it has a continuous version.   
\end{defn}
We denote by $L^{2}(\cS^{N})\coloneqq L^{2}(\Omega\times [0,T];\cS^{N})$, and we endow it with the norm

\begin{align}\label{normavol1}
\norma{Y}_{T,\alpha,\cS^{N}}\coloneqq \E\left(\int_{0}^{T}\exp(-\alpha t)\norma{Y(s)}_{\cS^{N}}^{2}\de s\right),\hskip 0,2cm Y\in L^{2}(\cS^{N}),
\end{align}
for some fixed $\alpha>0$ that will be fixed later on. Here, we denote by $\M_{{\rm ad}}^{2}(\cS^{N})\coloneqq L_{{\rm ad}}^{2}(\Omega\times [0,T];\cS^{N})$ the space of adapted processes with respect to the $\sigma$-algebra $(\bB(t)\otimes W^{\bQ}(t))_{t\in[0,T]}$ with values in $\cS^{N}$, and such that the norm \eqref{normavol1} is finite. Similarly, we define $\M_{{\rm ad}}^{2}(\hat{\Y}^{N})$.
\begin{rem}
We discuss the well-posedness of system \eqref{system2} for every choice of initial condition $(\X_{0},\blambda_{0})\in\hat{\Y}^{N}$. Furthermore, we are interested to looking forward the existence and uniqueness of solutions $\by\in \M_{{\rm ad}}^{2}(\hat{\Y}^{N})$.  
\end{rem}

\section{Main results}\label{section:mainresults}
In this section, we collect our main results. In the next Theorem, we study the well-posedness of \eqref{system2}.
\begin{thm}\label{importantprop}
Let us fix a filtered probability space $(\Omega,\B,\cP)$ endowed with a complete filtration $(\mathcal{F}_{t})_{t\in [0,T]}$ generated by $(\bB(t)\otimes W_{t}^{\bQ})_{t\in[0,T]}$. Assume that \autoref{A1}-\autoref{A8} hold true.  Then there exists $\alpha>0$ such that for every choice $\overline{\by}_{0}\coloneqq(\overline{\X}_{0},\overline{\blambda}_{0})$ as initial condition of \eqref{system2} such that $\overline{\by}_{0}\in L^{2}(\hat{\Y}^{N})$, the system \eqref{system2} has a unique solution $\by$ that belongs to $\M_{{\rm ad}}^{2}(\hat{\Y}^{N})$.
\end{thm}
In what follows, we state our propagation of chaos result.
\begin{thm}[Propagation of chaos]\label{lem:conv_muN_detailed}
Suppose that \autoref{A1}-\autoref{A8} hold true. Let $(\X(t),\blambda(t))$ be the unique strong solution to \eqref{system2} in $\M_{\rm ad}^2(\hat\Y^N)$ with initial condition $\overline{\by}_{0}\in L^{2}(\hat{\Y}^{N})$. Then there exists a unique measure-valued process
\begin{align}\label{misurasol}
\mu_{\cdot} \in C([0,T]; \cP_1(\Y)\cap \cP_{2}(\Y))    
\end{align}
satisfying the nonlinear McKean--Vlasov SDE
\begin{align}\label{eq:limitMV_detailed}
\begin{dcases}
\de Y(t) &= H(Y(t), \Lambda(t), \mu(t)) \, \de t + \sqrt{2\sigma_x} \, \de B(t),\\
\de \Lambda(t) &= \R(Y(t), \Lambda(t), \mu(t)) \, \de t + \sqrt{2\sigma_\lambda} \, \Pi_{\Lambda(t)} \, \de W^Q(t),\\
\mu(t)&={\rm law}(Y(t),\Lambda(t)),
\end{dcases}
\end{align}
with initial law $\mu_0 = \mathcal L(\overline{X}^1_0, \overline{\lambda}^1_0)$. Moreover, the empirical measure $\mu^N_t$ converges in expectation of the Wasserstein-1 distance to $\mu(t)$:
\begin{align}\label{eq:W1conv}
\lim_{N\rightarrow +\infty}\sup_{t \in [0,T]} \E \big[\W_1(\mu^N(t), \mu(t)) \big] =0.
\end{align}
In particular, the system \eqref{system2} is \emph{chaotic with limit law $\mu(t)$ in the sense of \autoref{defn:propchaos}}.
\end{thm}
\begin{thm}[Exponential convergence to equilibrium]\label{thm:exponentialconvergence}
Let us suppose that \autoref{A4},\autoref{A7}--\autoref{A12} hold true. Furthermore, suppose that the constants given in the hypotheses satisfy
\begin{align}
\kappa \coloneqq m_x + m_\lambda - 2L_\mu - \sigma_{\lambda}C_{Q}> 0,
\end{align}
and consider $\mu_{\cdot}$ as in \eqref{misurasol}. Then there exists a unique invariant probability measure $\mu_\infty \in \cP_{2}(\Y)$ such that for every initial condition $(Y(0),\Lambda(0))\in L^{2}(\Omega,\Y)$ of \eqref{eq:limitMV_detailed},
\begin{align}
\W_{2}(\mu(t), \mu_\infty)
\le
e^{-\kappa t} \W_{2}(\mu_0,\mu_\infty),
\end{align}
where $\mu_{0}=\L((Y(0),\Lambda(0)))$ is the law of $(Y(0),\Lambda(0))$.
\end{thm}
\begin{rem}
The exponential rate $\kappa$ depends explicitly on the structural constants appearing in 
\autoref{A4}--\autoref{A12}, namely on the monotonicity constants $m_x$, $m_\lambda$, 
the Lipschitz constant $L_\mu$, and the intensity of the noise through $\sigma_\lambda C_Q$. 
The condition $\kappa>0$ reflects a balance between the dissipativity of the drift fields 
and the destabilizing effect of the mean-field interaction and of the noise. From a conceptual viewpoint, this result is reminiscent of the exponential contraction estimates obtained for gradient flows in Wasserstein spaces. In the deterministic framework developed by Ambrosio--Gigli--Savaré, exponential convergence is typically derived from $\lambda$-convexity of the driving functional along Wasserstein geodesics and relies on deep geometric properties of the underlying metric measure space, such as lower Ricci curvature bounds and the validity of an Evolution Variational Inequality (EVI). In contrast, our analysis does not rely on any geometric curvature condition on the state space. The contraction estimate is instead obtained through direct monotonicity assumptions on the drift fields and the assumptions involving the noise operator. Thus, while the resulting exponential decay is formally similar, its mechanism is substantially different from the geometric one in the sense of \cite{ambrosio2014calculus}.
\end{rem}

\section{Proofs}\label{sec:proofs}
In this part, we prove our main results.
\begin{proof}[Proof of \autoref{importantprop}]
In what follows, we aim to prove this result by using the Picard iteration. Let $\by_0=(\overline{\X}_0,\overline{\blambda}_0)\in L^2(\hat{\Y}^N)$.
Define recursively, for $n\in\N$,
\begin{align}\label{Picard}
\begin{dcases}
\X_{n+1}(t)
= \overline{\X}_0
+ \int_0^t H(\X_n(s),\blambda_n(s),\mu_{n}^{N}(s))\,\de s
+ \sqrt{2\sigma_x}\,\bB(t),\\[0.25cm]
\blambda_{n+1}(t)
= \overline{\blambda}_0
+ \int_0^t\R(\X_n(s),\blambda_n(s),\mu_{n}^{N}(s))\,\de s\\
\qquad\qquad
+ \sqrt{2\sigma_\lambda}
\int_0^t\Pi_{\blambda_n(s)}\,\de W_s^{\bQ},
\end{dcases}
\end{align}
where
\begin{align}
 \mu_{n}^{N}(s)=\frac{1}{N}\sum_{i=1}^{N}\delta_{(X_{n}^{i}(s),\lambda_{n}^{i}(s))}.   
\end{align}
Notice that by using the same argument of \autoref{lem:positivemass}, one gets that $\blambda_{n}(t)\in (\cP(U))^{N}$ for all $n\in\N$ $\P$-a.s.. In what follows, we show that \eqref{Picard} has a unique strong solution by means of the Picard iteration. Fix $N\in\mathbb N$ and $T>0$. Let $(\Omega,\mathcal F,(\mathcal F_t)_{t\ge0},\mathbb P)$ be the filtered probability space supporting $(\bB\otimes W^{\bQ})$. For $\by=(\X,\blambda)\in\cS^N$ define recall that
\begin{align}
\|\by\|_{\cS^N}=\frac1N\sum_{i=1}^N\Big(\|\X^i\|_{2}+\|\blambda^i\|_{\mathrm{BL}}\Big).   
\end{align}
We consider as functional space $\M^2_{\mathrm{ad}}(0,T;\cS^N)$ endowed with the norm \eqref{normavol1}. Let us first prove that there exists a constant $r_{T}>0$, independent of $n$, such that
\begin{align}\label{daprov}
\sup_{t\in [0,T]}\sup_{n\in\mathbb N}
\E\|\by_n(t)\|_{\cS^N}^2 \le r_{T}.
\end{align}
Recall that for each $i=1,\dots,N$,
\begin{align}
\X_{n+1}^i(t)
=\overline \X_0^i
+\displaystyle\int_0^t H(\X_n^i(s),\blambda_n^i(s),\mu_n^N(s))\,\de s
+\sqrt{2\sigma_x}\,B^i(t).
\end{align}
Using $(a+b+c)^2\le 3(a^2+b^2+c^2)$ and Jensen's inequality, we obtain
\begin{align}
\E\|\X_{n+1}^i(t)\|_{2}^2
&\le
3\E\|\overline \X_0^i\|_{2}^2
+
3\E\Big\|\int_0^t H(\X_n^i(s),\blambda_n^i(s),\mu_n^N(s))\de s\Big\|_{2}^2
+
6\sigma_x \E\|B^i(t)\|_{2}^2.
\end{align}
By the linear growth assumption on $H$ (see \autoref{A3}), we get
\begin{align}
\|H(\X_n^i(s),\blambda_n^i(s),\mu_n^N(s))\|_{2}
&\le
M_{H}\big(1+\|\X_{n}^{i}(s)\|_{2}+\|\blambda_{n}^{i}(s)\|_{\rm BL}+m_1(\mu_{n}^{N}(s))\big).
\end{align}
Therefore,
\begin{align}
\E\Big\|\displaystyle\int_0^t H(\X_n^i(s),\blambda_n^i(s),\mu_n^N(s))\de s\Big\|_{2}^2
&\le
t\int_0^t \E\|H(\X_n^i(s),\blambda_n^i(s),\mu_n^N(s))\|_{2}^2\de s\\
&\le tM_{H}^{2}\displaystyle\int_{0}^{t}\E\big(1+\|\X_{n}^{i}(s)\|_{2}+\|\blambda_{n}^{i}(s)\|_{\rm BL}+m_1(\mu_{n}^{N}(s))\big)^{2}\de s\\
&\leq 6tM_{H}^{2}\displaystyle\int_{0}^{t}\left(1+\E\|\X_{n}^{i}(s)\|_{2}^{2}+\E\|\blambda_{n}^{i}(s)\|_{\rm BL}^{2}+ \E(m_1(\mu_{n}^{N}(s)))^{2}\right)\de s.
\end{align}
Then we have obtained that

\begin{align}\label{F1}
\begin{aligned}
 \E\|\X_{n+1}^i(t)\|_{2}^2&\leq 3\E\|\overline \X_0^i\|_{2}^2+ 6\sigma_{x}t\\
 &+18tM_{H}^{2}\displaystyle\int_{0}^{t}\left(1+\E\|\X_{n}^{i}(s)\|_{2}^{2}+\E\|\blambda_{n}^{i}(s)\|_{\rm BL}^{2}+ \E(m_1(\mu_{n}^{N}(s)))^{2}\right)\de s.
 \end{aligned}
\end{align}

Let us now recall that 

\begin{align}
\blambda_{n+1}^i(t)
=\overline\blambda_0^i+\displaystyle\int_0^t \R(\X_n^i(s),\blambda_n^i(s),\mu_n^N(s))\,\de s
+\sqrt{2\sigma_\lambda}\int_0^t \Pi_{\blambda_n^i(s)}\,\de W_s^{\bQ}.
\end{align}
 Then by \autoref{A5}, we have 

\begin{align}
 \E\norma{\blambda_{n+1}^i(t)}_{\BL}^{2}&\leq 3\E\norma{\overline\blambda_0^i}_{\BL}^{2}+ 
 3t\int_{0}^{t}\E\norma{\R(\X_n^i(s),\blambda_n^i(s),\mu_n^N(s))}_{\BL}^{2}\de s+ 6\sigma_{\lambda}\norma{\int_0^t \Pi_{\blambda_n^i(s)}\,\de W_s^{\bQ}}_{\BL}^{2}\\
 &\leq 3\E\norma{\overline\blambda_0^i}_{\BL}^{2}+ 3tM_{\R}^{2}\displaystyle\int_{0}^{t}\E(1+\norma{\X_{n}^{i}(s)}_{2}+m_{1}(\mu_{n}^{N}(s)))^{2}\de s\\
 &+6\sigma_{\lambda}\norma{\int_0^t \Pi_{\blambda_n^i(s)}\,\de W_s^{\bQ}}_{\BL}^{2}.
\end{align}
Further, as a consequence of \autoref{prop:R2}, we have that

\begin{align}
\E\left\|\displaystyle\int_0^t\Pi_{\blambda_n^{i}(s)}\de W_s^{\bQ}\right\|_{\BL}^2
\le
4C_{Q}
\E\int_0^t\left(1+\|\blambda_n^{i}(s)\|_{\BL}\right)^2 \de s \\
\leq 8C_{Q}
\E\int_0^t\left(1+\|\blambda_n^{i}(s)\|_{\BL}^{2}\right) \de s \\
\end{align}
where we have used that the operator norm $\norma{\Pi_{\lambda}}_{{\rm op}}\leq 1+\norma{\lambda}_{\BL}$. Then

\begin{align}\label{F2}
\begin{aligned}
\E\norma{\blambda_{n+1}^i(t)}_{\BL}^{2}&\leq  3\E\norma{\overline\blambda_0^i}_{\BL}^{2}+ 4C_{Q}
\E\int_0^t\|\blambda_n^{i}(s)\|_{\BL}^2 \de s\\
&+6tM_{\R}^{2}\displaystyle\int_{0}^{t}(1+\E\norma{\X_{n}^{i}(s)}_{2}^{2}+\E(m_{1}(\mu_{n}^{N}(s)))^{2})\de s.
\end{aligned}
\end{align}
Combining \eqref{F1}, and \eqref{F2}, and letting $C\coloneqq \max\{6\sigma_{x}T, 18T^{2}M_{H}^{2},8C_{Q}(1+T), 6T^{2}M_{\R}^{2}\}$

\begin{align}
\E\|\X_{n+1}^i(t)\|_{2}^2 + \E\norma{\blambda_{n+1}^i(t)}_{\BL}^{2}&\leq 3\E\|\overline \X_0^i\|_{2}^2+ 3\E\norma{\overline\blambda_0^i}_{\BL}^{2}+ 6\sigma_{x}t\\
 &+18tM_{H}^{2}\displaystyle\int_{0}^{t}\left(1+\E\|\X_{n}^{i}(s)\|_{2}^{2}+\E\|\blambda_{n}^{i}(s)\|_{\rm BL}^{2}+ \E(m_1(\mu_{n}^{N}(s)))^{2}\right)\de s+\\
 &+4C_{Q}
\E\int_0^t\|\blambda_n^{i}(s)\|_{\BL}^2 \de s\\
&+6tM_{\R}^{2}\displaystyle\int_{0}^{t}(1+\E\norma{\X_{n}^{i}(s)}_{2}^{2}+\E(m_{1}(\mu_{n}^{N}(s)))^{2})\de s\\
&\leq 3\E\|\overline \X_0^i\|_{2}^2+ 3\E\norma{\overline\blambda_0^i}_{\BL}^{2}+ 3C+\\
&+ 2C\displaystyle\int_{0}^{t}\left(\E\norma{\X_{n}^{i}(s)}_{2}^{2}+\|\blambda_n^{i}(s)\|_{\BL}^2+\E(m_{1}(\mu_{n}^{N}(s)))^{2}\right)\de s.
\end{align}

On the other hand, we notice that 

\begin{align}
\E(m_{1}(\mu_{n}^{N}(s)))^{2}\leq\frac{2}{N}\sum_{i=1}^{N}\left(\E\norma{\X_{n}^{i}(s)}_{2}^{2}+\E\norma{\by_{n}^{i}(s)}_{\BL}^{2}\right)   
\end{align}
and thus,

\begin{align}
\frac{1}{N}\sum_{i=1}^{N}\left(\E\|\X_{n+1}^i(t)\|_{2}^2 + \E\norma{\blambda_{n+1}^i(t)}_{\BL}^{2}\right)&\leq \frac{3}{N}\sum_{i=1}^{N}\left(\E\|\overline \X_0^i\|_{2}^2+\E\norma{\overline\blambda_0^i}_{\BL}^{2}\right)+3C\\
&+\frac{6C}{N}\displaystyle\int_{0}^{t}\sum_{i=1}^{N}\left(\E\|\X_{n}^i(s)\|_{2}^2 + \E\norma{\blambda_{n}^i(s)}_{\BL}^{2}\right)\de s.
\end{align}
Define

\begin{align}\label{useddopo}
M_n(t)\coloneqq \frac{1}{N}\sum_{i=1}^{N}\left(\E\|\X_{n}^i(t)\|_{2}^2 + \E\norma{\blambda_{n}^i(t)}_{\BL}^{2}\right).  
\end{align}
Then
\begin{align}
\sup_{n\in\N}M_{n+1}(t)\le \frac{3}{N}\sum_{i=1}^{N}\left(\E\|\overline \X_0^i\|_{2}^2+\E\norma{\overline\blambda_0^i}_{\BL}^{2}\right)+3C
+6C\displaystyle\int_0^t \sup_{n\in\N}M_n(s)\de s.
\end{align}
and thus by Gr\hol{o}nwall we obtain 

\begin{align}
\sup_{t\in [0,T]}\sup_{n\in\N}M_{n}(t)\leq \left(\frac{3}{N}\sum_{i=1}^{N}\left(\E\|\overline \X_0^i\|_{2}^2+\E\norma{\overline\blambda_0^i}_{\BL}^{2}\right)+3C\right)e^{6CT}   
\end{align}
and from here, we conclude that \eqref{daprov} holds true. Let us now estimate $\W_{1}(\mu_{n}^{N}(t),\mu_{n-1}^{N}(t))$. Define the coupling
\begin{align}
\gamma_t=\frac1N\sum_{i=1}^N
\delta_{(\by_n^i(t),\by_{n-1}^i(t))}.
\end{align}
Then
\begin{align}
\W_1(\mu_{n}^{N}(t),\mu_{n-1}^{N}(t))
\le \int \|y-\tilde y\|_{\cS}\,\de\gamma_t=\frac{1}{N}\sum_{i=1}^N\|\by_n^i(t)-\by_{n-1}^i(t)\|_{\cS}.   
\end{align}

Hence
\begin{align}\label{boundwass}
\W_1(\mu_{n}^{N}(t),\mu_{n-1}^{N}(t))\le\|\by_n(t)-\by_{n-1}(t)\|_{\cS^N}.
\end{align}
From \eqref{Picard},
\begin{align}
\X_{n+1}(t)-\X_n(t)
=\displaystyle\int_0^t \Big(H(\X_n,\blambda_n,\mu_n^N)-H(\X_{n-1},\blambda_{n-1},\mu_{n-1}^N)\Big)\de s,    
\end{align}
since the Brownian term cancels. Similarly,

\begin{align}
\blambda_{n+1}(t)-\blambda_n(t)
&=\displaystyle\int_0^t\Big(\R(\X_n,\blambda_n,\mu_n^N)-\R(\X_{n-1},\blambda_{n-1},\mu_{n-1}^N)\Big) \de s\\
&+\sqrt{2\sigma_\lambda}\displaystyle\int_0^t\Big(\Pi_{\blambda_n(s)}-\Pi_{\blambda_{n-1}(s)}\Big)\de W_s^{\bQ}.
\end{align}
Let us define $\delta\by_n := \by_{n+1}-\by_n$, and let $r\coloneqq r_{T}$ as determined by \eqref{daprov}. Notice that \autoref{A1}-\autoref{A2}, one gets that for each $i=1,\ldots, N$,

\begin{align*}
&\|H(\X_n^{i},\blambda_n^{i},\mu_n^N) - H(\X_{n-1}^{i},\blambda_{n-1}^{i},\mu_{n-1}^N) \|_2\\
&\leq \|H(\X_n^{i},\blambda_n^{i},\mu_n^N) - H(\X_{n-1}^{i},\blambda_{n-1}^{i},\mu_{n}^N) \|_2+\|H(\X_{n-1}^{i},\blambda_{n-1}^{i},\mu_{n}^N) - H(\X_{n-1}^{i},\blambda_{n-1}^{i},\mu_{n-1}^N) \|_2\\
&\le C_{H,r} \| \delta \by_n^{i} \|_{\cS} + C_{H,r} \W_1(\mu_n^N, \mu_{n-1}^N) \\
&\le 2 C_{H,r} \| \delta \by_n \|_{\cS^N}.
\end{align*}
Thus
\begin{align}
\|\delta\X_{n+1}^{i}(t)\|_{2}\le 2C_{H,r}\displaystyle\int_0^t\|\delta\by_n(s)\|_{\cS^N}ds.    
\end{align}
Squaring and using Jensen inequality:

\begin{align}
\|\delta\X_{n+1}^{i}(t)\|_{2}^2\le4C_{H,r}^2 t\int_0^t\|\delta\by_n(s)\|_{\cS^{N}}^2 ds.    
\end{align}
Let us now consider the deterministic part of $\delta\blambda_{n+1}$. Similarly, notice that by \autoref{A6},
\begin{align*}
&\| \R(\X_n^{i}(s), \blambda_n^{i}(s), \mu_n^N(s)) - \R(\X_{n-1}^{i}(s), \blambda_{n-1}^{i}(s), \mu_{n-1}^N(s)) \|_{\rm BL} \\
&\le 
\| \R(\X_n^{i}, \blambda_n^{i}, \mu_n^N) - \R(\X_{n-1}^{i}, \blambda_{n-1}^{i}, \mu_n^N) \|_{\rm BL} 
+ \| \R(\X_{n-1}^{i}, \blambda_{n-1}^{i}, \mu_n^N) - \R(\X_{n-1}^{i}, \blambda_{n-1}^{i}, \mu_{n-1}^N) \|_{\rm BL} \\
&\le L_{\R,r} \| \delta \by_n^{i}(s) \|_{\cS} + L_{\R,r} \W_1(\mu_n^N(s), \mu_{n-1}^N(s)) \\
&\le 2 L_{\R,r} \| \delta \by_n(s) \|_{\cS^N}.
\end{align*}
On the other hand, notice that

\begin{align}\label{itobound1}
\E\left\|\displaystyle\int_0^t(\Pi_{\blambda_n}-\Pi_{\blambda_{n-1}})\de W_s^{\bQ}\right\|_{\BL}^2
\le C_{Q}
\E\int_0^t\|\delta\blambda_n(s)\|_{\BL}^2\de s.   
\end{align} 
Therefore,

\begin{align}
&\E\|\delta\by_{n+1}(t)\|_{\cS^{N}}^2=\E\left(\frac{1}{N}\sum_{i=1}^{N}\|\delta\by_{n+1}^{i}(t)\|_{\cS}\right)^{2}\\
&\phantom{formu}=E\left(\frac{1}{N}\sum_{i=1}^{N}\|\delta\X_{n+1}^{i}(t)\|_{2}+ \|\delta\by_{n+1}^{i}(t)\|_{\BL}\right)^{2}\\
&\phantom{formu}\leq \E\left(2C_{H,r}\displaystyle\int_0^t\|\delta\by_n(s)\|_{\cS^N}ds + 2 L_{\R,r} \| \delta \by_n(s) \|_{\cS^N}+ \left\|\displaystyle\int_0^t(\Pi_{\blambda_n}-\Pi_{\blambda_{n-1}})\de W_s^{\bQ}\right\|_{\BL}\right)^{2}\\
&\phantom{formu}\leq 6t(C_{H,r}^{2}+L_{\R,r}^{2}+ C_{Q}^{2})\displaystyle\int_0^t\|\delta\by_n(s)\|_{\cS^N}^{2}\de s
\end{align}
where in the last inequality we have used $(a+b+c)^{2}\leq 3a^{2}+3b^{2}+3c^{2}$, and Jensen inequality. Let us define $\rho\coloneqq 6T(C_{H,r}^{2}+L_{\R,r}^{2}+ C_{Q}^{2})$. Notice that we have obtained that 

\begin{align}
\int_0^T e^{-\alpha t} \E \|\delta \by_{n+1}(t)\|_{\cS^N}^2 \de t 
&\le \rho \int_0^T e^{-\alpha t} \int_0^t \E \|\delta \by_n(s)\|_{\cS^N}^2 \de s \de t \\
&= \rho \int_0^T \E \|\delta \by_n(s)\|_{\cS^N}^2 \int_s^T e^{-\alpha t} \de t \de s \\
&\le \frac{\rho}{\alpha} \int_0^T e^{-\alpha s} \E \|\delta \by_n(s)\|_{\cS^N}^2 \de s \\
&= \frac{\rho}{\alpha} \norma{\delta \by_n}_{T,\alpha,\cS^{N}}^2.
\end{align}
Then by iterating the previous argument we have that 
\begin{align}
\|\delta\by_n\|_{T,\alpha,\cS^{N}}^2\le\left(\frac{\rho}{\alpha}\right)^n\|\delta\by_0\|_{T,\alpha,\cS^{N}}^2.   
\end{align}
Notice that this inequality can be written as 
\begin{align}
\|\delta\by_n\|_{t,\alpha,\cS^{N}}^2\le\left(\frac{\rho}{\alpha}\right)^n\|\delta\by_0\|_{t,\alpha,\cS^{N}}^2.   
\end{align}
for all $0<t\leq T$. Let us now define
\begin{align}
A_n\coloneqq \Big\{ \sup_{t \in [0,T]} \|\delta \by_n\|_{t,\alpha,\cS^{N}} > \varepsilon_n \Big\}, 
\qquad 
\varepsilon_n \coloneqq \Big(\frac{\rho}{\alpha}\Big)^{n/4}.
\end{align}

\noindent
By Markov's inequality
\begin{align}
\P(A_n) 
&= \mathbb{P}\Big( \sup_{t \in [0,T]} \|\delta \by_n\|_{t,\alpha,\cS^{N}} > \varepsilon_n \Big) \\
&\le \frac{\E\Big[\sup_{t\in[0,T]} \|\delta \by_n\|_{t,\alpha,\cS^{N}}^2\Big]}{\varepsilon_n^2} \\
&\leq \frac{\|\delta \by_n\|_{T,\alpha,\cS^{N}}^2}{\varepsilon_n^2} \\
&\le \frac{(\rho/\alpha)^n \|\delta \by_0\|_{T,\alpha,\cS^{N}}^2}{(\rho/\alpha)^{n/2}} \\
&= (\rho/\alpha)^{n/2} \, \|\delta \by_0\|_{T,\alpha,\cS^{N}}^2.
\end{align}
In what follows, we can choose $\alpha>0$ such that $0<\frac{\rho}{\alpha}<1$ in such a way that

\begin{align}
\sum_{n=1}^{\infty} \mathbb{P}(A_n) \le \|\delta \by_0\|_{T,\alpha,\cS^{N}}^2 \sum_{n=1}^{\infty} (\rho/\alpha)^{n/2} < \infty,.  
\end{align}
Thus by Borel-Cantelli, 

 \begin{align}
\mathbb{P}\Big( \limsup_{n\to\infty} A_n \Big) = 0.     
 \end{align}
Hence

\begin{align}
\sup_{t \in [0,T]} \|\delta \by_n\|_{t,\alpha,\cS^{N}} \rightarrow 0 \hskip 0,2cm \text{a.s..} 
\end{align}
Let $\Delta=\by-\tilde\by$ such that $\by(0)=\tilde\by(0)$. Notice that by the previous argument, we have that

\begin{align}
\E\|\Delta(t)\|_{\cS^{N}}^2\le \rho\int_0^t\E\|\Delta(s)\|_{\cS^{N}}^2 \de s.   
\end{align}
Then by Gr\hol{o}nwall inequality
\begin{align}
 \E\|\Delta(t)\|_{\cS^{N}}^2\leq \E\|\Delta(0)\|_{\cS^{N}}^2 e^{\rho t}.   
\end{align}
Since $\Delta(0)=0$, we are done. Notice that by \eqref{boundwass}, we have that $\{\mu_{n}^{N}(t)\}_{n\in\N}$ is a Cauchy sequence. Since by \cite[Proposition 2.2.8]{panaretos2020} $(\cP_{1}(\cS^{N}),\W_{1})$ is complete, then $\mu_{n}^{N}(t)$ converges weakly to $\mu^{N}(t)$. Hence, we can pass to the limit in \eqref{Picard} to obtain that existence and uniqueness of a solution to \eqref{system2} in $\cS^{N}$. Since we need to find a solution in $\hat{\Y}^{N}$, we need to check that $\blambda\in (\cP(U))^{N}$. Indeed, Since $\{\blambda_{n}(t)\}_{n\in\N}$ is a Cauchy sequence with respect to the ${\rm BL}$-norm $\norma{\blambda_{n}(t)}_{\rm BL}=\frac{1}{N}\sum_{i=1}^{N}\norma{\lambda_{n}^{i}}_{\rm BL}$, by applying \eqref{disutile}, for any $\varepsilon>0$ there exists $N\in \N$ such that 

\begin{align*}
 \W_{1}(\blambda_{n}(t),\blambda_{m}(t))\leq (1+D_{U})\norma{\blambda_{n}(t)-\blambda_{m}(t)}_{\rm BL}\leq (1+D_{U})\varepsilon 
\end{align*}
for any $n,m\geq N$ positive integers, and all $t\in [0,T]$. On the other hand, for some $\overline{x}\in U$,

\begin{align*}
m_{1}(\blambda_{n}(t))=\frac{1}{N}\sum_{i=1}^{N}\int_{U}\de_{U} (x,\overline{x})\lambda_{n}^{i}(\de x)&\leq \sup_{x\in U}\de_{U} (x,\overline{x})\\
&\leq {\rm diam}(U)<+\infty,
\end{align*}
where ${\rm diam}(U)$ is the diameter of $U$. Then, we conclude that $\{\blambda_{n}(t)\}_{n\in\N}\}$ is a Cauchy sequence with respect to the Wasserstein-1 distance. Again by \cite[Proposition 2.2.8]{panaretos2020}, $((\cP_{1}(U))^{N},\W_{1})$ is a complete metric space, then $\blambda(t)\in (\cP(U))^{N}$, and we are done.
\end{proof}
In what follows, we prove \autoref{lem:conv_muN_detailed}.
\begin{proof}[Proof of \autoref{lem:conv_muN_detailed}]
Let us first define the path space of continuous trajectories
\begin{align}
\cY \coloneqq C([0,T]; \cS) = C([0,T]; \Tt^d \times \F(U)).    
\end{align}
We endow $\cY$ with the uniform norm

\begin{align}
 \|\by\|_{\cY} \coloneqq \sup_{t \in [0,T]} \| (X_t, \lambda_t) \|_{\cS} = \sup_{t \in [0,T]} \big( \|X_t\|_2 + \|\lambda_t\|_{\rm BL} \big)   
\end{align}
and we consider the space of probability measures $\cP_1(\cY)$ endowed with the Wasserstein-1 distance

\begin{align}
\W_1(\nu_1, \nu_2) \coloneqq \inf_{\pi \in \Pi(\nu_1, \nu_2)}\displaystyle\int_{\cY \times \cY} \|\by - \tilde \by\|_{\cY} \, \pi(\de\by, \de\tilde \by)    
\end{align}
where $\Pi(\nu_1, \nu_2)$ denotes the set of couplings of $\nu_1$ and $\nu_2$. We now for each $i=1,\dots,N$, let $(Y^i(t), \Lambda^i(t))$ solve the nonlinear McKean--Vlasov SDE
\begin{align}\label{eq:limitMV_detailed_repeat}
\begin{cases}
\de Y^i(t) = H(Y^i(t), \Lambda^i(t), \mu(t))\, \de t + \sqrt{2\sigma_x}\, \de B^i_t,\\
\de\Lambda^i(t) = \R(Y^i(t), \Lambda^i(t), \mu(t))\, \de t + \sqrt{2\sigma_\lambda} \, \Pi_{\Lambda^i(t)}\, \de W^{Q,i}(t),
\end{cases}
\end{align}
with initial condition $(Y^i_0, \Lambda^i_0) = (X^i_0, \lambda^i_0)$ and independent Brownian motion $(B^i)$, and independent noise $(W^{Q,i})$ which are pairwise independent.  Denote their common law by
\begin{align}
\mu = \mathcal{L}((Y^i(t), \Lambda^i(t))_{t \in [0,T]}),
\end{align}
let us prove that $\mu\in \cP_1(\cY)$, or equivalently
\begin{align}
\mu_\cdot \in C([0,T];\cP_1(\Y)).
\end{align}
Let $(Y,\Lambda)$ solve \eqref{eq:limitMV_detailed_repeat}. For $0\le s\le t\le T$ we have
\begin{align}
Y(t)-Y(s) &= \int_s^t H(Y(r),\Lambda(r),\mu(r))\,\de r + \sqrt{2\sigma_x}\,(B(t)-B(s)),\\
\Lambda(t)-\Lambda(s) &= \int_s^t \R(Y(r),\Lambda(r),\mu(r))\,\de r + \sqrt{2\sigma_\lambda}\int_s^t \Pi_{\Lambda(r)}\,\de W^Q(r).
\end{align}
With a similar argument as the one used to estimate \eqref{useddopo}, we can that defined

\begin{align}
\hat{M}_N(t)\coloneqq \frac{1}{N}\sum_{i=1}^{N}\left(\E\|Y^i(t)\|_{2}^2 + \E\norma{\Lambda^i(t)}_{\BL}^{2}\right), 
\end{align}
we get
\begin{align}\label{usedopo2}
\sup_{t\in [0,T]}\hat{M}_{N}(t)\leq \left(\frac{3}{N}\sum_{i=1}^{N}\left(\E\|Y_0^i\|_{2}^2+\E\norma{\Lambda_0^i}_{\BL}^{2}\right)+3C\right)e^{6CT}   
\end{align}
for some $C>0$ independent of $N$. Then there exists a radius $r'>0$ such that $\mu(t)\in \cP_{1}(B_{r'}^{\Y})$. Using Lipschitz continuity of $H$ and $\R$ and the triangle inequality, we get
\begin{align}
\|Y(t)-Y(s)\|_2 &\le \int_s^t \|H(Y(r),\Lambda(r),\mu(r)) - H(Y(s),\Lambda(s),\mu(s))\|_2 \, \de r + \int_s^t \|H(Y(s),\Lambda(s),\mu(s))\|_2 \, \de r,\\
\|\Lambda(t)-\Lambda(s)\|_{\rm BL} &\le \int_s^t \|\R(Y(r),\Lambda(r),\mu(r)) - \R(Y(s),\Lambda(s),\mu(s))\|_{\rm BL} \, \de r + \int_s^t \|\R(Y(s),\Lambda(s),\mu(s))\|_{\rm BL} \, \de r.
\end{align}
Applying the Lipschitz property (with constants $C_{H,r'}, L_{\R,r'}$), we have
\begin{align}
\|H(Y(r),\Lambda(r),\mu(r)) - H(Y(s),\Lambda(s),\mu(s))\|_2 &\le C_{H,r'} \big(\|Y(r)-Y(s)\|_2 + \|\Lambda(r)-\Lambda(s)\|_{\rm BL} + \W_1(\mu(r),\mu(s)) \big),\\
\|\R(Y(r),\Lambda(r),\mu(r)) - \R(Y(s),\Lambda(s),\mu(s))\|_{\rm BL} &\le L_{\R,r'} \big(\|Y(r)-Y(s)\|_2 + \|\Lambda(r)-\Lambda(s)\|_{\rm BL} + \W_1(\mu(r),\mu(s)) \big).
\end{align}
Further, we have 
\begin{align}
\int_s^t \|\R(Y(s),\Lambda(s),\mu(s))\|_{\rm BL} \, \de r \leq M_{H}(t-s)\left(1+\norma{(Y(s),\Lambda(s))}_{\cS}+ m_{1}(\mu(s))\right). 
\end{align} 

For the stochastic integral, we have
\begin{align}
\E \Big\| \int_s^t \Pi_{\Lambda(r)}\,\de W^Q(r) \Big\|_{\rm BL}^2 \le C_Q \int_s^t \E \|\Pi_{\Lambda(r)}\|_{\rm BL}^2 \, \de r \le 2C_Q \int_{s}^{t}(1+\E\norma{\Lambda(r)}_{\BL}^{2})\de r.
\end{align}
Define
\begin{align}
\Delta(s,t) := \E \big[\|Y(t)-Y(s)\|_2 + \|\Lambda(t)-\Lambda(s)\|_{\rm BL} \big].    
\end{align}
Then we get
\begin{align}
\Delta(s,t) &\le  C_{H,r'}\int_s^t \Delta(s,r) \, \de r + L_{\R,r'}\int_s^t  \Delta(s,r) \, \de r + \sqrt{2\sigma_x d (t-s)} + \\
&+\sqrt{4\sigma_\lambda C_Q (t-s)}\left(\int_{s}^{t}(1+\E\norma{\Lambda(r)}_{\BL}^{2})\de r\right)^{\frac{1}{2}}+2\int_{s}^{t}\W_{1}(\mu(r),\mu(s))\de r.
\end{align}
Hence, we have obtained that

\begin{align}
\Delta(s,t) &\le C_{H,r'} \int_s^t \Delta(s,r) \, \de r + L_{R,r'}\int_s^t  \Delta(s,r) \, \de r + \sqrt{2\sigma_x d (t-s)} + \\
&+C_{T}'(t-s)+2\int_{s}^{t}\Delta(s,r)\de r,
\end{align}
where $C_{T}'>0$ is a positive constant depending on $C_{Q},\sigma_{\lambda}, C$, and the initial conditions but not on $N$. Applying Grönwall’s inequality (for the integral term), we conclude
\begin{align}
\Delta(s,t) \le \big(\sqrt{2\sigma_x d (t-s)} + C_{T}' (t-s)\big) \exp\big((2+C_{H,r'}+L_{R,r'})(t-s)\big) \end{align}
Therefore, we have
\begin{align}
\W_1(\mu(t),\mu(s)) \le \Delta(s,t)\to 0 \quad \text{as } |t-s| \to 0.
\end{align}
Hence $t\mapsto \mu(t)$ is continuous in $\W_1$. Furthermore, let us set, 

\begin{align}
\overline{\mu}^{N}=\frac{1}{N}\sum_{i=1}^{N}\delta_{(Y^i(t), \Lambda^i(t))}.
\end{align}
and let us define
\begin{align}
\gamma^N = \frac{1}{N} \sum_{i=1}^N \delta_{((X^i, \lambda^i), (Y^i, \Lambda^i))} \in \cP_1(\cY \times \cY).
\end{align}
Similarly to the $\W_{1}$, we consider the $\W_{2}$-distance as 

\begin{align}
\W_2(\nu_1, \nu_2) \coloneqq \inf_{\pi \in \Pi(\nu_1, \nu_2)}\displaystyle\int_{\cY \times \cY} \|\by - \tilde \by\|_{\cY}^{2} \, \pi(\de\by, \de\tilde \by),    
\end{align}
so that, by the very definition
\begin{align}\label{eq:W2_coupling_detailed}
\begin{aligned}
\W_{2}(\mu^N, \overline{\mu}^{N}) &\le \int_{\cY \times \cY} \|\by - \tilde \by\|_\cY^{2} \, \gamma^N(d\by, d\tilde\by)\\
&= \frac{1}{N} \sum_{i=1}^N \sup_{t\in[0,T]} \|(X^i(t), \lambda^i(t)) - (Y^i(t), \Lambda^i(t))\|_\cS^{2}.\\
&\leq\frac{2}{N} \sum_{i=1}^N \sup_{t\in[0,T]}\left(\norma{X^{i}(t)-Y^{i}(t)}_{2}^{2}+\norma{\lambda^{i}(t)-\Lambda^{i}(t)}_{\BL}^{2}\right).
\end{aligned}
\end{align}
For each $i$ and $t \in [0,T]$, define the difference
\begin{align}
\delta X^i(t) = X^i(t) - Y^i(t), \quad \delta \lambda^i(t) = \lambda^i(t) - \Lambda^i(t).    
\end{align}
Then, from \eqref{system2} and \eqref{eq:limitMV_detailed_repeat}, we have
\begin{align*}
\delta X^i(t) &= \displaystyle\int_0^t \big(H(X^i(s), \lambda^i(s), \mu^N(s)) - H(Y^i(s), \Lambda^i(s), \mu(s))\big) \de s,\\
\delta \lambda^i(t) &= \displaystyle\int_0^t \big(\R(X^i(s), \lambda^i(s), \mu^N(s)) - \R(Y^i(s), \Lambda^i(s), \mu(s))\big) \de s
+ \sqrt{2\sigma_\lambda} \displaystyle\int_0^t (\Pi_{\lambda^i(s)} - \Pi_{\Lambda^i(s)}) \, \de W^{Q,i}(s).
\end{align*}
By assumptions \autoref{A1}-\autoref{A2}, \autoref{A5},
\begin{align}
\|H(X^i(s), \lambda^i(s), \mu^N(s)) - H(Y^i(s), \Lambda^i(s), \mu(s))\|_2 &\le C_{H,r}\big( \|\delta X^i(s)\|_2 + \|\delta \lambda^i(s)\|_{\BL}+ \W_{1}(\mu^{N}(s),\mu(s))\big),\\
\|\R(X^i(s), \lambda^i(s), \mu^N(s)) - \R(Y^i(s), \Lambda^i(s), \mu(s))\|_{\BL} &\le L_{\R,r} \big( \|\delta X^i(s)\|_2 + \|\delta \lambda^i(s)\|_{\BL}+\W_{1}(\mu^{N}(s),\mu(s))\big).
\end{align}
Moreover, recall that by \autoref{A8} and \autoref{prop:R2}, we have \eqref{itobound1}, that is,
\begin{align}
\E \Big\| \displaystyle\int_0^t (\Pi_{\lambda^i(s)} - \Pi_{\Lambda^i(s)}) \, \de W^{Q,i}(s) \Big\|_{\rm BL}^2 \le C_Q \displaystyle\int_0^t \E \|\delta \lambda^i(s)\|_{\rm BL}^2 \, \de s.   
\end{align}
Let us define
\begin{align}
\Delta^i(t)\coloneqq\E \sup_{t\in[0,T]}\left(\norma{X^{i}(t)-Y^{i}(t)}_{2}^{2}+\norma{\lambda^{i}(t)-\Lambda^{i}(t)}_{\BL}^{2}\right).   
\end{align}
Notice that
\begin{align}
\frac{1}{N} \sum_{i=1}^N \Delta^i(t) \le  (3C_{H,r}^{2}T+3L_{\R,r}^{2}T+ 2\sigma_{\lambda}C_{Q})\left(\displaystyle\int_0^t\frac{1}{N} \sum_{i=1}^N \Delta^i(s)\, \de s+ \int_{0}^{T} \E(\W_{1}(\mu^{N}(s),\mu(s)))^{2}\de s \right).  
\end{align}
Let us set $C_{T}\coloneqq 4C_{H,r}^{2}T+4L_{\R,r}^{2}T+ 2\sigma_{\lambda}C_{Q}$. We have obtained that

\begin{align}
\frac{1}{N} \sum_{i=1}^N \Delta^i(t) \le C_{T}\left(\displaystyle\int_0^t\frac{1}{N} \sum_{i=1}^N \Delta^i(s)\, \de s+ \int_{0}^{T} \E(\W_{1}(\mu^{N}(s),\mu(s)))^{2}\de s \right).
\end{align}
This inequality can be improved for $s\in [0,t]$ as

\begin{align}
\frac{1}{N} \sum_{i=1}^N \Delta^i(s) \le C_{T}\left(\displaystyle\int_0^s\frac{1}{N} \sum_{i=1}^N \Delta^i(r)\, \de r+ \int_{0}^{t} \E(\W_{1}(\mu^{N}(r),\mu(r)))^{2}\de r \right).
\end{align}
Then by Gr\hol{o}nwall inequality, we get

\begin{align}
\frac{1}{N} \sum_{i=1}^N \Delta^i(s) \leq C_{T}e^{C_{T}s}\int_{0}^{t} \E(\W_{1}(\mu^{N}(r),\mu(r)))^{2}\de r.   
\end{align}
Hence

\begin{align}
    \W_{2}^{2}(\mu^{N}(t),\overline{\mu}^{N}(t))\leq \frac{2}{N} \sum_{i=1}^N \Delta^i(s) \leq C_{T}e^{C_{T}T}\int_{0}^{t} \E(\W_{1}(\mu^{N}(r),\mu(r)))^{2}\de r,
\end{align}
and since $\W_{1}\leq \W_{2}$, we have

\begin{align}
\W_{1}^{2}(\mu^{N}(t),\mu(t))&\leq 2\W_{2}^{2}(\mu^{N}(t),\overline{\mu}^{N}(t)) + 2\W_{2}^{2}(\overline{\mu}^{N},\mu)\\
&\leq 2C_{T}e^{C_{T}T}\int_{0}^{t} \E(\W_{1}(\mu^{N}(r),\mu(r)))^{2}\de r +2\W_{2}^{2}(\overline{\mu}^{N},\mu),
\end{align}
and from which

\begin{align}
\E\W_{1}^{2}(\mu^{N}(t),\mu(t))\leq 2\exp(2C_{T}e^{C_{T}T}T)\E\W_{2}^{2}(\overline{\mu}^{N},\mu). 
\end{align}
Notice that by applying the same reasoning as in \eqref{useddopo}, we obtain that

\begin{align}
  m_{2}(\overline{\mu}^{N})<C_{T}  
\end{align}
where $C_{T}$ is a positive constant independent of $N$, and the bound holds true for $\mu$. Furthermore, since $\overline{\mu}^{N}$ is composed by i.i.d random measures, and $\overline{\mu}^{N}\rightarrow \mu$ weakly with $m_{2}(\overline{\mu}^{N})\rightarrow m_{2}(\mu)$, then by \cite[Theorem 2.2.1]{panaretos2020} we have that 

\begin{align}
  \W_{2}(\overline{\mu}^{N},\mu)\rightarrow 0 \hskip 0,5cm \P-a.s..  
\end{align}
Now, notice that 

\begin{align}
 \E\W_{2}^{2}(\overline{\mu}^{N},\mu)\leq 2m_{2}(\overline{\mu}^{N})+ 2m_{2}(\mu)<+\infty   
\end{align}
uniformly in $N$. So that dominate convergence, we have 

\begin{align}
 \lim_{N\rightarrow+\infty}\E\W_{2}^{2}(\overline{\mu}^{N},\mu)=0   
\end{align}
and thus

\begin{align}
\lim_{N\rightarrow +\infty}\sup_{t \in [0,T]} \E \big[\W_1(\mu^N(t), \mu(t)) \big] =0.
\end{align}
From here the system \eqref{system2} is chaotic with limit law $\mu_{\cdot}$ in the sense of \autoref{defn:propchaos}.
\end{proof}
We now prove \autoref{thm:exponentialconvergence}.
\begin{proof}[Proof of \autoref{thm:exponentialconvergence}]
Let $\mu(t),\nu(t)$ be two solutions and consider a synchronous coupling 
$(Y^1(t),\Lambda^1(t))$, $(Y^2(t),\Lambda^2(t))$ driven by the same noises:
\begin{align*}
&\begin{cases}
\de Y^1(t) = H(Y^1(t), \Lambda^1(t), \mu(t))\, \de t + \sqrt{2\sigma_x}\, \de B(t),\\
\de\Lambda^1(t) = \R(Y^1(t), \Lambda^1(t), \mu(t))\,\de t + \sqrt{2\sigma_\lambda}\, \Pi_{\Lambda^1(t)}\, \de W^{Q}_t,
\end{cases} 
\\
\\
&\begin{cases}
dY^2(t) = H(Y^2(t), \Lambda^2(t), \nu_t)\, \de t + \sqrt{2\sigma_x}\, \de B(t),\\
d\Lambda^2(t) = \R(Y^2(t), \Lambda^2(t), \nu(t))\, \de t + \sqrt{2\sigma_\lambda}\, \Pi_{\Lambda^2(t)}\, \de W^{Q}(t).
\end{cases}
\end{align*}
Now consider the difference
\begin{align}
&\Delta Y(t) = Y^1(t) - Y^2(t), \hskip 0,2cm \Delta \Lambda(t) = \Lambda^1(t) - \Lambda^2(t),\\
&\Delta H(t)= H(Y^1(t), \Lambda^1(t), \mu(t))-H(Y^2(t), \Lambda^2(t), \nu_t)\\
&\Delta \R(t)=\R(Y^1(t), \Lambda^1(t), \mu(t))-\R(Y^2(t), \Lambda^2(t), \nu(t)).
\end{align}
and define
\begin{align}
\Phi(t)=\E\!\left[\|\Delta Y(t)\|_{2}^2+\|\Delta\Lambda(t)\|_{\BL}^2\right].    
\end{align}
Notice that

\begin{align}
&\de\Delta Y(t)=\Delta H(t)\de t\\
&\de\Delta\Lambda(t)=\Delta\R(t)\,\de t+\sqrt{2\sigma_\lambda}(\Pi_{\Lambda^1(t)}-\Pi_{\Lambda^2(t)})\de W^Q(t). 
\end{align}
Now we obtain that
\begin{align}
\frac{\de}{\de t}\Phi(t)
&=2\E\nm{\Delta Y(t),\Delta H(t)}+2\E\nm{\Delta\Lambda(t),\Delta\R(t)}\nonumber\\
&\quad+2\sigma_\lambda C_{Q}\E\|\Lambda^1(t)-\Lambda^2(t)
\|_{\BL}^2.
\end{align}
Decompose

\begin{align}
\Delta H(t)=[H(Y^{1}(t),\Lambda^{1}(t),\mu(t))-H(Y^{2}(t),\Lambda^{2}(t),\mu(t))]+[H(Y^{2}(t),\Lambda^{2}(t),\mu(t))-H(Y^{2}(t),\Lambda^{2}(t),\nu(t))].    
\end{align}
Using uniform convexity \autoref{A10}--\autoref{A12}

\begin{align}
\E\nm{\Delta Y(t),\Delta H(t)}&\le- m_x\E\|\Delta Y(t)\|_{2}^2
+L_\mu \W_{1}(\mu(t),\nu(t))\sqrt{\E\|\Delta Y(t)\|_{2}^2}\\
&\leq - m_x\E\|\Delta Y(t)\|_{2}^2+ \frac{L_{\mu}}{2}\W_{1}^{2}(\mu(t),\nu(t))+ \frac{L_{\mu}}{2}\E\|\Delta Y(t)\|_{2}^2
\end{align}
Similarly,

\begin{align}
\E\nm{\Delta\Lambda(t),\Delta\R(t)}&\le- m_\lambda\E\|\Delta\Lambda(t)\|_{\BL}^2
+L_\mu W_{1}(\mu(t),\nu(t))\sqrt{\E\|\Delta\Lambda(t)\|_{\BL}^2}\\
&- m_\lambda\E\|\Delta\Lambda(t)\|_{\BL}^2+\frac{L_{\mu}}{2}\W_{1}^{2}(\mu(t),\nu(t))+ \frac{L_{\mu}}{2}\E\|\Delta\Lambda(t)\|_{\BL}^2.
\end{align}
Since $\W_{1}\leq \W_{2}$, and $W_2^2(\mu_t,\nu_t)\le\Phi(t)$, we then have

\begin{align}
\E\langle \Delta Y_t,\Delta H_t\rangle
+
\E\langle \Delta\Lambda_t,\Delta\R_t\rangle
\le
-(m_x+m_\lambda-2L_\mu)\Phi(t).    
\end{align}
Therefore
\begin{align}
\frac{\de}{\de t}\Phi(t)
\le-2(m_x+m_\lambda-2L_\mu)\Phi(t)+2\sigma_\lambda C_Q \Phi(t)\eqqcolon-2\kappa\Phi(t).   
\end{align}
By Grönwall,
\begin{align}
\Phi(t)\le e^{-2\kappa t}\Phi(0).
\end{align}
Since $\Phi(0)=\W_2^2(\mu_0,\nu_0)$ and $\Phi(t)\ge \W_2^2(\mu(t),\nu(t))$,
\begin{align}\label{imp}
\W_2(\mu(t),\nu(t))
\le
e^{-\kappa t}\W_2(\mu_0,\nu_0).    
\end{align}
Let us consider
\begin{align}\label{mapphibar}
\hat{\Phi} : C([0,T];\cP_2(\Y)) \longrightarrow C([0,T];\cP_2(\Y))
\end{align}
by
\begin{align}
(\hat{\Phi}(\mu_\cdot))(t) \coloneqq {\rm law}(Y(t),\Lambda(t)),
\quad t\in[0,T].
\end{align}
For $t>s$, and the uniqueness of solution to the SDEs the semigroup property holds:
\begin{align}
\mu(t)=(\hat{\Phi}\mu(s))(t-s).    
\end{align}
Notice that the evolution between $\mu_s$ and $\mu_0$, can be estimated as

\begin{align}
\W_2(\mu(t),\mu(s))=\W_2((\hat{\Phi}\mu(s))(t-s),(\hat{\Phi}\mu_0)(t-s))\le e^{-\kappa(t-s)}\W_2(\mu(s),\mu_0).  
\end{align}
By letting $s$ fixed, we have $m_{2}((Y(r),\Lambda(r)))<+\infty$ for all $r\in [0,s]$, and 

\begin{align}
\limsup_{t\to\infty}W_2(\mu(t),\mu(s))=0.   
\end{align}
Letting $s\to\infty$ yields

\begin{align}
\lim_{t,s\to\infty}\W_2(\mu(t),\mu(s))=0.    
\end{align}
Thus $(\mu_t)$ is Cauchy in $C([0,\infty),(\cP_2(\Y),W_2))$. Since $(\cP_2(\Y),W_2)$ is complete,
there exists $\mu_\infty$ such that
\begin{align}
\lim_{t\rightarrow +\infty}\W_{2}(\mu(t),\mu_{\infty})=0.   
\end{align}
Notice that for any $h>0$, and any test function $\phi$

\begin{align}
\nm{(\hat{\Phi}\mu_\infty)(h),\phi}=\lim_{t\to\infty}\nm{(\hat{\Phi}\mu(t))(h),\phi}=\lim_{t\to\infty}\nm{\mu(t+h),\phi}=\nm{\mu_\infty,\phi},   
\end{align}
so $\mu_{\infty}$ is invariant. Let us set $\nu(t)=(\hat{\Phi}\mu_{\infty})(t)$, and $\mu(t)=(\hat{\Phi}\mu_{0})(t)$. Then by \eqref{imp}, and the invariance of $\mu_{\infty}$, we have

\begin{align}
\W_2(\mu(t),\mu_{\infty})
\le
e^{-\kappa t}\W_2(\mu_0,\mu_{\infty}),    
\end{align}
and we are done.
\end{proof}
\section{Future Directions}
The present work opens several natural directions for further investigation. In this paper, the long-time and mean-field analysis is carried out under structural assumptions on the covariance operator $Q$ ensuring well-posedness of the projected stochastic integral. 
A natural and challenging extension concerns the study of the associated nonlinear continuity (or Fokker--Planck) equation satisfied by the law $\mu(t)$ when the operator $Q$ is arbitrary and not necessarily of finite trace. In this more general setting, the measure-valued component is driven by a cylindrical noise that may generate genuinely infinite-dimensional diffusion effects.  Understanding the well-posedness of the corresponding continuity equation requires a careful analysis of the induced second-order operator in the space of probability measures, as well as suitable coercivity or dissipativity properties. 
In particular, it would be of interest to establish existence, uniqueness, and stability results for the limiting equation without compactness or trace-class assumptions on $Q$. Another natural question concerns the intrinsic well-posedness of the nonlinear evolution equation satisfied by $\mu(t)$ in the absence of strong structural assumptions on the drift fields. This includes identifying minimal conditions ensuring global existence, preservation of positivity, and stability in Wasserstein spaces. 
A variational or gradient-flow formulation in the space of probability measures could provide a promising framework for such an analysis. A further direction consists in replacing the Gaussian cylindrical perturbation with a Lévy-type noise. In this case, the measure-valued component would evolve according to a jump-driven stochastic dynamics projected onto the tangent space of $\cP(U)$. Such a setting would lead to a nonlinear nonlocal equation at the mean-field level, combining transport, diffusion, and jump terms. Establishing well-posedness, propagation of chaos, and long-time behavior in this framework represents a substantial analytical challenge and would significantly broaden the applicability of the present model.
\appendix
\section{Random fields indexed by the parameter space}\label{sec:frechet}
We recall the basic notions of differentiability in Banach spaces which are used throughout the paper. Let $E,F$ be Banach spaces and let $\Phi:E\to F$ be a map.
\begin{defn}[Fr\'echet derivative]
We say that $\Phi$ is \emph{Fr\'echet differentiable} at a point $x\in E$ if there exists a bounded linear operator $D\Phi(x):E\to F$ such that

\begin{align}
\lim_{\norma{h}_{E}\to0}\frac{\norma{\Phi(x+h)-\Phi(x)-D\Phi(x)h}_{F}}{\norma{h}_{E}}=0.    
\end{align}
The operator $D\Phi(x)$ is called the \emph{Fr\'echet derivative} of $\Phi$ at $x$.
\end{defn}
In the next, we recall what a Gaussian measure is on a generic Banach space $E$.

\begin{defn}
Let $E$ be a generic Banach space, and denote by $\mathscr{B}(E)$ the Borel $\sigma$-field of $E$. A probability measure $\gamma$ on $(E,\mathscr{B}(E))$ is said to be Gaussian if $\gamma \circ f^{-1}$ is Gaussian measure in $\R$ for every $f\in E^{\ast}$ (the topological dual space of $E$ of linear and continuous functionals $f:E\rightarrow \R$). The measure is called centered (or symmetric) if all the measures $\gamma\circ f^{-1}$ are centered.
\end{defn}
\begin{rem}
Notice that $f\in E^{\ast}$ then $f\in L^{p}(X,\gamma)$ for all $p\geq 1$. In fact, we have that
\begin{align}
\displaystyle\int_{E}\vert f(x) \vert^{p}\gamma(\de x)=\displaystyle\int_{\R}\vert t\vert^{p}(\gamma\circ f^{-1})(\de t)   
\end{align}
where the right-hand side of the previous equality is finite due to $\gamma\circ f^{-1}$ is a Gaussian measure in $\R$.
\end{rem}
In what follows, we give a suitable characterization of Gaussian measures.

\begin{defn}
Let $\gamma$ be a Gaussian measure on $(E,\mathscr{B}(E))$. We define the mean $a_{\gamma}$, and the covariance $B_{\gamma}$ of $\gamma$ by
\begin{align}
&a_{\gamma}(f)\coloneqq \displaystyle\int_{E}f(x)\gamma(\de x), \hskip 0,1cm f\in E^{\ast},\\
& B_{\gamma}(f,g)\coloneqq \displaystyle\int_{E}\left[ f(x)-a_{\gamma}(f)\right]\left[g(x)-a_{\gamma}(g)\right]\gamma(\de x), \hskip 0,1cm f,g\in E^{\ast}.
\end{align}
\end{defn}
Notice that $f\mapsto a_{\gamma}(f)$ is linear and $(f,g)\mapsto B_{\gamma}(f,g)$ is bilinear in $E^{\ast}$. Furthermore, 
\begin{align}
  B_{\gamma}(f,f)=\norma{f-a_{\gamma}(f)}_{L^{2}(E,\gamma)}^{2}\geq 0,\hskip 0,2cm f\in E^{\ast}.
\end{align}
In what follows, we denote by $\hat{\gamma}$ the characteristic function of $\gamma$ which is defined as

\begin{align}
 \hat{\gamma}(f)\coloneqq \displaystyle\int_{E}\exp\{if(x)\}\gamma(\de x), \hskip 0,2cm f\in E^{\ast}. 
\end{align}

\begin{thm}\label{characterization:gaussians}
Let $\gamma$ be a Gaussian measure on $(E,\mathscr{B}(E))$. Then
\begin{align}
\hat{\gamma}(f)=\exp\left(ia_{\gamma}(f)-\frac{1}{2}B_{\gamma}(f,f)\right),\hskip 0,2cm f\in E^{\ast}.  
\end{align}
Therefore the law of $f$ in $\Tt$ is given by

\begin{align}
   \gamma\circ f^{-1}=\mathscr{N}(a_{\gamma}(f),B_{\gamma}(f,f)).
\end{align}
Conversely, if a Borel probability measure $\gamma$ is such that
\begin{align}
\hat{\gamma}(f)=\exp\left(ia(f)-\frac{1}{2}B(f,f)\right),\hskip 0,2cm f\in E^{\ast}  
\end{align}
for some linear operator $a: E^{\ast}\rightarrow \Tt$, and some bilinear symmetric nonnegative operator $B: E^{\ast}\times E^{\ast}\rightarrow \Tt$, then $\gamma$ is a Gaussian measure on $(E,\mathscr{B}(E))$ with mean $a$, and covariance operator $B$.
\end{thm}
\section*{Declarations}

\subsection*{Ethical Approval}
Not applicable. This study involves theoretical mathematical research and does not involve human subjects, animal experiments, or experimental data requiring ethical approval.

\subsection*{Funding}
A.M.H. has been supported by project PRIN 2022 ``understanding the LEarning process of QUantum Neural networks (LeQun)'', proposal code 2022WHZ5XH -- CUP J53D23003890006. 

\subsection*{Availability of Data and Materials}
Not applicable. The theoretical results and mathematical proofs presented in this manuscript are based on analytical derivations. No datasets were generated in this study.
\section*{Acknowledgments}
AMH would like to thank Giuseppe D'Onofrio for inspiring discussions that greatly influenced this work.

\bibliographystyle{unsrt}
\bibliography{bibliography}

\end{document}